\documentclass[12pt]{amsart}
\usepackage[osf,sc]{mathpazo}
\usepackage{amssymb}

\usepackage{geometry}
\usepackage{bbm}
\usepackage{graphicx}
\usepackage{hyperref}
\hypersetup{
    colorlinks=true, %set true if you want colored links
    linktoc=all,     %set to all if you want both sections and subsections linked
    linkcolor=blue,
        citecolor=red,
    filecolor=black,
    urlcolor=blue	%choose some color if you want links to stand out
}
\usepackage{enumerate}
\usepackage[inline]{enumitem}
\makeatletter
\newcommand{\inlineitem}[1][]{%
\ifnum\enit@type=\tw@
    {\descriptionlabel{#1}}
  \hspace{\labelsep}%
\else
  \ifnum\enit@type=\z@
       \refstepcounter{\@listctr}\fi
    \quad\@itemlabel\hspace{\labelsep}%
\fi} \makeatother
\newcommand{\gn}{\nu}

\newcommand{\gf}{\phi}

\newcommand{\gch}{\chi}

\newcommand{\Gd}{\Delta}

\newcommand{\subs}{\subset}
\newcommand{\sups}{\supset}

\newcommand{\bs}{\backslash}

\newcommand{\ti}{\tilde}

\newcommand{\mbb}{\mathbb}

\newcommand{\mcl}{\mathcal}

\newcommand{\us}{\underset}
\newcommand{\os}{\overset}

\newcommand{\lra}{\longrightarrow}

\newcommand{\N}{\mbb N}

\newcommand{\R}{\mcl R}

\newcommand{\Ra}{\Rightarrow}

\newcommand{\es}{\emptyset}

\newcommand{\equ}[1]{%
\begin{equation*}
#1
\end{equation*}
}
\newcommand{\equa}[1]{%
\begin{equation*}
\begin{aligned}
#1
\end{aligned}
\end{equation*}
}
\newcommand{\equan}[2]{%
\begin{equation}
\label{Eq:#1}
\begin{aligned}
#2
\end{aligned}
\end{equation}
}
\DeclareMathOperator{\Det}{Det}

\newcommand{\vmattwo}[4]{%
\begin{vmatrix}
  #1 & #2\\ #3 & #4
\end{vmatrix}
}

\newcommand{\vmatthree}[9]{%
\begin{vmatrix}
  #1 & #2 & #3\\ #4 & #5 & #6\\ #7 & #8 & #9
\end{vmatrix}
}

\newtheorem{theorem}{Theorem}[section]

\newtheorem{lemma}[theorem]{Lemma}

\theoremstyle{definition}
\newtheorem{defn}[theorem]{Definition}
\newtheorem{example}[theorem]{Example}

\theoremstyle{remark}
\newtheorem{note}[theorem]{Note}
\newtheorem{remark}[theorem]{Remark}
\numberwithin{equation}{section}
\makeatletter
\def\namedlabel#1#2{\begingroup
   \def\@currentlabel{#2}%
   \label{#1}\endgroup
}
\makeatother
\newtheorem*{thmA}{\bf{Theorem A}}

\begin{document}
\title[On the Enumeration of a Certain Type of Hyperplane Arrangements]{On the Enumeration of a Certain Type of Hyperplane Arrangements}
\author[C.P. Anil Kumar]{Author: C.P. Anil Kumar*}
\address{Center for Study of Science, Technology and Policy
\# 18 \& \#19, 10th Cross, Mayura Street,
Papanna Layout, Nagashettyhalli, RMV II Stage,
Bengaluru - 560094
Karnataka,INDIA
}
\email{akcp1728@gmail.com}
\thanks{*The author is supported by a research grant and facilities provided by Center for study of Science, 
Technology and Policy (CSTEP), Bengaluru, INDIA for this research work.}
\date{\sc \today}
\begin{abstract}
In this article we prove in the main theorem that, there is a bijection between the isomorphism classes of a certain type of real hyperplane arrangements on the one hand, and the antipodal pairs of convex cones of an associated discriminantal arrangement on the other hand. The type of hyperplane arrangements considered and the isomorphism classes have been defined precisely. As a consequence we enumerate such isomorphism classes by computing the characteristic polynomial of the discriminantal arrangement. With a certain restriction, the enumerated value is shown to be independent of the discriminantal arrangement. Later we observe that the restriction we impose on the type of hyperplane arrangements is a mild one and that this conditional restriction is quite generic. Moreover the restriction is defined in terms of a normal system being  concurrency free which is a generic condition. We also discuss two examples of normal systems which are not concurrency free in the last section and enumerate the number of isomorphism classes.
\end{abstract}
\maketitle
\section{\bf{Introduction and a brief survey}}
Enumeration of line arrangements in the real projective plane up to an equivalence has been studied by W.~B. Carver~\cite{MR0005632} subsequent to his work on systems of linear inequalities~\cite{MR1502612}.
The notion of equivalence defined in ~\cite{MR0005632} is in terms of regions (see the definition on Page 674). In~\cite{MR0005632} on page $674$, it is mentioned that the problem of finding how many non-equivalent figures $F_n$ (an arrangement of 
$n$-lines) in the real projective plane, exist, for large $n$ is still unanswered. For some initial values of $n$ a complete list of representatives for equivalence classes is known. The initial seven values for $1\leq n\leq 7$ are given by $1,1,1,1,1,4,11$. For the next four values ($8\leq n\leq 11$) refer to Sloane's OEIS A018242 at \url{https://oeis.org/A018242} or J.~E.~Goodman, J.~O'Rourke, C.~D.~T\'{o}th~\cite{MR3793131}, Chapter 5, Section 5.6, Page 146, Table 5.6.1, second row on simple arrangements of n lines. In this article we consider Euclidean arrangements instead of arrangements in a projective space.

Concurrency Geometries have been studied by H.~H.~Crapo in~\cite{MR0766269} in order to solve problems on configurations of hyperplane arrangements in a Euclidean space. Imagine a finite set $\mcl{H}_n^m$ of $n$-hyperplanes in $\mbb{R}^m$ for some $n,m\in \N$ which can move freely and whose normal directions arise from a fixed finite set $\mcl{N}$ of cardinality $n$ which is generic, that is, any subset $\mcl{M}\subseteq \mcl{N}$ of cardinality at most $m$ is linearly independent. They give rise to different hyperplane arrangements in $\mbb{R}^m$. This paper is devoted to the combinatorial study of such configurations of hyperplane arrangements. In particular we answer a question of enumerating equivalence classes of this type of hyperplane arrangements over $\mbb{R}^m,m\in \mbb{N}$, though in principle, the field $\mbb{R}$ can be replaced by an ordered field $\mbb{F}$. 
The precise definitions and the equivalence notion to describe the isomorphism classes of hyperplane arrangements are given below before stating the main theorem. 

This main theorem leads to interesting combinatorics. A brief survey after remark~\ref{Remark:Methodology} and a new view point in the more relevant Section~\ref{sec:CCA}, describing the combinatorial aspects are mentioned after stating the main theorem.

\begin{remark}
For those who are interested in generalising to ordered fields other than the subfields of $\mbb{R}$, I would like to mention that the basic theory of ordered fields is given in N.~Jacobson~\cite{MR0780184}~(Chapter $5$),\cite{MR1009787}~(Chapter $11$) and S.~Lang~\cite{MR1878556} (Chapter $11$).
\end{remark}

\section{\bf{Definitions and statement of the main result}}
\label{sec:StatementMainResult}
We begin the section with a few definitions.
\begin{defn}[A Hyperplane Arrangement, A Generic Hyperplane Arrangement, A Central Hyperplane Arrangement]
	\label{defn:HA}
	~\\
	Let $m,n$ be positive integers. We say a set 
	\equ{\mcl{H}_n^m=\{H_1,H_2,\ldots,H_n\}} of $n$ affine hyperplanes in $\mbb{R}^m$ forms a 
	hyperplane arrangement. We say that they form a generic hyperplane arrangement or a hyperplane arrangement in general position, if Conditions 1,2 are satisfied.
	\begin{itemize}
		\item Condition 1: For $1\leq r \leq m$, the intersection of any $r$ hyperplanes has dimension $m-r$.
		\item Condition 2: For $r>m$, the intersection of any $r$ hyperplanes is empty.
	\end{itemize}
We say that the hyperplane arrangement $\mcl{H}_n^m$ is central if $\us{i=1}{\os{n}{\cap}}H_i\neq \es$.
\end{defn}
Since this article is about enumeration of a certain type of isomorphism classes of hyperplane arrangements, a basic enumeration result (Theorem~\ref{theorem:NumberofRegions}) in this subject, is mentioned . 
\begin{note}
	Any finite set of $n$-hyperplanes in $\mbb{R}^m$ divides the Euclidean space into finitely many regions which may be bounded or unbounded (in the usual sense). We define these non-empty regions as polyhedral regions. 
	There are $2^n$ choices of inequalities for describing the regions. However only a few of the choices of inequalities give rise to non-empty regions as given by the following theorem whose proof is well known in the literature on hyperplane arrangements. 
	Refer to R.~Stanley~\cite{MR2868112}, Proposition $3.11.8$ and his notes on Page $347$ for literature. Also refer to R.~C.~Buck~\cite{MR0009119}.
\end{note}
In this article, from now on, a polyhedral region means a non-empty polyhedral region.
\begin{theorem}
	\label{theorem:NumberofRegions}
	Let $n,m$ be positive integers and $\mcl{H}_n^m$ be a generic hyperplane arrangement. Then there are
	\begin{itemize}
		\item $\us{i=0}{\os{m}{\sum}}\binom{n}{i}$ polyhedral regions, 
		\item $\binom{n-1}{m}$ bounded polyhedral regions and
		\item $\us{i=0}{\os{m-1}{\sum}}\binom{n}{i}+\binom{n-1}{m-1}$ unbounded polyhedral regions.
	\end{itemize} 
\end{theorem}
 \begin{defn}[Normal System]
 	\label{defn:NS}
 	~\\
 Let $\mcl{N}=\{L_1,L_2,\ldots,L_n\}$ be a finite set of lines passing through the origin in $\mbb{R}^m$. Let $\mcl{U}=\{\pm v_1,\pm v_2,\ldots,\pm v_n\}$ be a set of antipodal pairs of non-zero vectors on these lines. We say that
 	$\mcl{N}$ forms a normal system, if the set 
 	\equ{\mcl{B}=\{v_1,v_2,\ldots,v_n\}}
 	of vectors has the property that, any subset $\mcl{C}\subs \mcl{B}$ of cardinality at most $m$ is a linearly independent set. 
 \end{defn}
\begin{defn}[Hyperplane arrangement given by a normal system]
	\label{defn:HAGivenbyNS}
	~\\
	Let $\mcl{N}=\{L_1,L_2,\ldots,L_n\}$ be a normal system in $\mbb{R}^m$. Let 
	$\mcl{U}=\{\pm (a_{i1},a_{i2},\ldots,a_{im})\mid 0\neq (a_{i1},a_{i2},\ldots,a_{im})\in L_i, 1\leq i\leq n\}$ 
	be a set of antipodal pairs of vectors of the normal system $\mcl{N}$. We fix the  matrix $[a_{ij}]_{1\leq i\leq n,1\leq j\leq m}\in M_{n\times m}(\mbb{R})$. Let $\mcl{H}_n^m=\{H_1,H_2,\ldots,H_n\}$ 
	be any hyperplane arrangement in $\mbb{R}^m$ whose equations are given by 
	\equ{H_i:\us{j=1}{\os{m}{\sum}} a_{ij}x_j=b_i \text{ for some }b_i\in \mbb{R}.}
	We say that the hyperplane arrangement $\mcl{H}_n^m$ is given by the normal system $\mcl{N}$.
\end{defn}
\begin{defn}[Normal System Associated to a Generic Hyperplane Arrangement]
	\label{defn:NSAHA}
	~\\
	Let $\mcl{H}^m_n=\{H_i:\us{j=1}{\os{m}{\sum}}a_{ij}x_j=b_i,1\leq i \leq n\}$ be a generic hyperplane arrangement. Then the normal system $\mcl{N}$
	associated to $\mcl{H}^m_n$ is given by
	\equ{\mcl{N}=\{L_i=\{t(a_{i1},a_{i2},\ldots,a_{im})\in \mbb{R}^m\mid t\in \mbb{R}\}\mid 1\leq i \leq n\}}
	and a set of antipodal pairs of normal vectors is given by 
	\equ{\mcl{U}=\{\pm v_1,\ldots,\pm v_n\}} where $0\neq v_i\in L_i, 1\leq i \leq n$.
	For example we can choose by default 
	\equ{\mcl{U}=\{\pm (a_{i1},a_{i2},\ldots,a_{im})\in \mbb{R}^m\mid 1\leq i \leq n\}.}
\end{defn}

\begin{defn}[Isomorphism Between Two Generic Hyperplane Arrangements]
	\label{defn:Iso}
	Let
	\equ{(\mcl{H}_n^m)_1=\{H^1_1,H^1_2,\ldots,H^1_n\},
		(\mcl{H}_n^m)_2=\{H^2_1,H^2_2,\ldots,H^2_n\}}
	be two generic hyperplane arrangements in $\mbb{R}^m$. We say a map 
	$\gf:(\mcl{H}_n^m)_1 \lra (\mcl{H}_n^m)_2$
	is an isomorphism between these two generic hyperplane arrangements if $\gf$ is a bijection between the sets
	$(\mcl{H}_n^m)_1,(\mcl{H}_n^m)_2$, in particular on the subscripts $1\leq i\leq n$ satisfying the following property: given 
	$1\leq i_1<i_2<\ldots<i_{m-1}\leq n$ and lines \equ{L=H^1_{i_1}\cap H^1_{i_2}\cap \ldots \cap H^1_{i_{m-1}},
		M=H^2_{\gf(i_1)}\cap H^2_{\gf(i_2)}\cap \ldots \cap H^2_{\gf(i_{m-1})},}
	the order of vertices, that is, zero dimensional intersections on the lines $L,M$, agrees via the bijection induced by $\gf$ again 
	on the sets of subscripts of cardinality $m$ (corresponding to the vertices on $L$) containing 
	$\{i_1,i_2,\ldots,i_{m-1}\}$ and those (corresponding to the vertices on $M$) containing 
	$\{\gf(i_1),\gf(i_2),\ldots,\gf(i_{m-1})\}$. There are four possibilities of pairs of orders and any one
	pairing of orders out of these four pairs must agree via the map induced by $\gf$. We say the isomorphism $\gf$ preserves subscripts or $\gf$ is trivial on subscripts if in addition to being an isomorphism it satisfies $\gf(H_i^1)=H_i^2$ for $1\leq i\leq n$.
\end{defn}
\begin{note}
	If there is an isomorphism between two generic hyperplane arrangements $(\mcl{H}_n^m)_i,i=1,2$,
	then there exists a piecewise linear bijection of $\mbb{R}^m$ to $\mbb{R}^m$ which takes one arrangement to 
	another, using suitable triangulation of the polyhedral regions. For obtaining a piecewise linear isomorphism 
	extension from vertices to the one-dimensional skeleton of the arrangements, further subdivision is not needed.
\end{note}
Now we define an arrangement of hyperplanes which is not a generic hyperplane arrangement but is a central arrangement which in the literature is known as the discriminantal arrangement or the Manin-Schechtman arrangement (refer to Page 205, Section 5.6 in P.~Orlik and H.~Terao~\cite{MR1217488}). Some of the authors who have worked on the discriminantal arrangements are  C.~A.~Athanasiadis~\cite{MR1720104}, M.~Bayer and K.~Brandt~\cite{MR1456579}, M.~Falk~\cite{MR1209098}, Yu.~I.~Manin and V.~V.~Schechtman~\cite{MR1097620} and more recently A.~Libgober and S.~Settepanella~\cite{MR3899551}.  
We mention the definition here.
\begin{defn}[Discriminantal Arrangement-A Central Arrangement]
	\label{defn:CA}
	~\\
	Let \equ{\mcl{H}_n^m=\{H_1,H_2,\ldots,H_n\}} 
	be a generic hyperplane arrangement of $n$ hyperplanes in $\mbb{R}^m$. Let the equation for $H_i$ be given by 
	\equ{\us{j=1}{\os{m}{\sum}}a_{ij}x_j=b_i,\text{ with } a_{ij}, b_i\in \mbb{R}, 1\leq j\leq m, 1\leq i\leq n.}
	For every $1\leq i_1<i_2<\ldots<i_{m+1}\leq n$ consider the hyperplane $M_{\{i_1,i_2,\ldots,i_{m+1}\}}$ 
	passing through the origin in $\mbb{R}^n$ in the variables $y_1,y_2,\ldots,y_n$ whose equation is given 
	by 
	\equ{\Det
		\begin{pmatrix}
			a_{i_11} & a_{i_12} & \cdots & a_{i_1(m-1)} & a_{i_1m} & y_{i_1}\\
			a_{i_21} & a_{i_22} & \cdots & a_{i_2(m-1)} & a_{i_2m} & y_{i_2}\\
			\vdots   & \vdots   & \ddots & \vdots       & \vdots   & \vdots\\
			a_{i_{m-1}1} & a_{i_{m-1}2} & \cdots & a_{i_{m-1}(m-1)} & a_{i_{m-1}m} & y_{i_{m-1}}\\
			a_{i_m1} & a_{i_m2} & \cdots & a_{i_m(m-1)} & a_{i_mm} & y_{i_m}\\
			a_{i_{m+1}1} & a_{i_{m+1}2} & \cdots & a_{i_{m+1}(m-1)} & a_{i_{m+1}m} & y_{i_{m+1}}\\
		\end{pmatrix}
		=0}
	Then the associated discriminantal arrangement of hyperplanes passing through the origin in $\mbb{R}^n$ 
	is given by
	\equ{\mcl{C}^n_{\binom{n}{m+1}}=\{M_{\{i_1,i_2,\ldots,i_{m+1}\}}\mid 1\leq 
		i_1<i_2<\ldots<i_{m+1}\leq n\}.}
	It is a central arrangement consisting of hyperspaces, that is, linear subspaces of codimension one in $\mbb{R}^n$. 
\end{defn}

\begin{note}
	Even though the definition of hyperplanes $M_{\{i_1,i_2,\ldots,i_{m+1}\}}$ of the discriminantal arrangement $\mcl{C}^n_{\binom{n}{m+1}}$ involves the coefficients $[a_{ij}]_{1\leq j\leq m,1\leq i\leq n}$ of the variables $x_i,1\leq i\leq m$ 
	we can pick and fix any one set of equations for the hyperplanes $H_i, 1\leq i \leq n$ of the hyperplane arrangement to associate the discriminantal arrangement.
\end{note}
\begin{note}
	In general, the normal lines of these hyperplanes $M_{\{i_1,i_2,\ldots,i_{m+1}\}}$ of the discriminantal arrangement $\mcl{C}^n_{\binom{n}{m+1}}$ need not form a normal system. However they will be distinct, as they correspond to different subsets of $\{1,2,\ldots,n\}$ of cardinality $m+1$. In the equation of any hyperplane $M_{\{i_1,i_2,\ldots,i_{m+1}\}}$ of the discriminantal arrangement $\mcl{C}^n_{\binom{n}{m+1}}$,
	there are $(m+1)$ non-zero coefficients and the rest are zero coefficients.
\end{note}
\begin{note}[Convention: Fixing the coefficient matrix of any hyperplane arrangement for a fixed given normal system]
	\label{note:Convention}
	Let $\mcl{N}=\{L_1,L_2,\ldots,L_n\}$ be a normal system in $\mbb{R}^m$. Let $\mcl{U}=\{\pm (a_{i1},a_{i2}$,  $\ldots,a_{im})\mid (a_{i1},a_{i2},\ldots,a_{im})\in L_i, 1\leq i\leq n\}$ 
	be a set of antipodal pairs of vectors of the normal system $\mcl{N}$. We fix the  matrix $[a_{ij}]_{1\leq i\leq n,1\leq j\leq m}\in M_{n\times m}(\mbb{R})$. 
	Let $\mcl{H}_n^m=\{H_1,H_2,\ldots,H_n\}$ be any hyperplane arrangement with the normal system $\mcl{N}$. When we write equations for the hyperplane $H_i$, we use the fixed matrix and write 
	\equ{H_i:\us{j=1}{\os{m}{\sum}} a_{ij}x_j=b_i \text{ for some }b_i\in \mbb{R}.}
	Using this coefficient matrix, we define the discriminantal arrangement which depends only on the normal system. Two hyperplane arrangements with the same normal system give two points
	$(b_1,b_2,\ldots,b_n),(c_1,c_2,\ldots,c_n)$. If these vectors lie in the interior of the same cone of the discriminantal arrangement, then the hyperplane arrangements are generic, and they are isomorphic by an isomorphism which is 
	trivial on subscripts. In general, if the arrangements are isomorphic by such an isomorphism we say $(b_1,b_2,\ldots,b_n)$ is isomorphic to $(c_1,c_2,\ldots,c_n)$. For example, $(b_1,b_2,\ldots,b_n)$
	is isomorphic to $-(b_1,b_2,\ldots,b_n)$ even though they lie in opposite cones.
\end{note}
\begin{note}
	We note that regions of the discriminantal arrangement 
	$\mcl{C}^n_{\binom{n}{m+1}}$ are all convex conical, unbounded and there are 
	at most $\us{i=0}{\os{n}{\sum}}\binom{\binom{n}{m+1}}{i}-\binom{\binom{n}{m+1}-1}{n}$
	such regions using Theorem~\ref{theorem:NumberofRegions}.
\end{note}
We prove a certain property of a general discriminantal arrangement arising from a normal system in the lemma below.
With the notations of Definition~\ref{defn:CA} and convention in Note~\ref{note:Convention}, if we fix any $m$ variables (say) $y_1,y_2,\ldots,y_m$ out of $n$ variables $y_i,1\leq i\leq n$ in the discriminantal arrangement $\mcl{C}^n_{\binom{n}{m+1}}$ and give them real values $c_1,c_2,\ldots,c_m$, 
then we can solve for the variable $y_i$ as the value $c_i$ for $m+1\leq i\leq n$ in terms of values $c_1,c_2,\ldots,c_m$. Then the solution we get $(c_1,c_2,\ldots,c_n)$ automatically satisfies all the $\binom{n}{m+1}$ equations. 
We state this as a lemma below.
\begin{lemma}[Lemma on the Geometry of Concurrencies]
	\label{lemma:Intersection}
	~\\
	With the notations of Definition~\ref{defn:CA} and convention in Note~\ref{note:Convention}, let $\mcl{C}^n_{\binom{n}{m+1}}$ be the discriminantal arrangement. Let $1\leq i_1<i_2<\ldots<i_m\leq n$.
	The solutions for $y_{j},\  j\neq i_k,1\leq k\leq m$ in terms of $y_{i_1},y_{i_2},\ldots,y_{i_m}$ satisfies all 
	the equations of the discriminantal arrangement.
\end{lemma}
\begin{proof}
	If $(y_1,y_2,\ldots,y_n)=(c_1,c_2,\ldots,c_n)$ is a solution to the subset of the equations where 
	$y_j=c_j,\  j\neq i_k,1\leq k\leq m$ is expressible in terms of the variables $y_{i_k}=c_{i_k},\ 1 \leq k\leq m$, then by 
	the very definition of the discriminantal arrangement we obtain that the hyperplanes
	\equ{H^{c_{i_1}}_{i_1},H^{c_{i_2}}_{i_2},\ldots,H^{c_{i_m}}_{i_m},H^{c_j}_j}
	concur for every $j\neq i_k,1\leq k\leq m$. Here the equation for $H_i^c,1\leq i\leq n$ is given by $\us{j=1}{\os{m}{\sum}}a_{ij}x_j=c$.
	Hence any set of $(m+1)$ hyperplanes 
	\equ{H^{c_{j_1}}_{j_1},H^{c_{j_2}}_{j_2},\ldots,H^{c_{j_m}}_{j_m},H^{c_{j_{m+1}}}_{j_{m+1}}}
	concur. This proves the lemma.
\end{proof}
\begin{note}
	\label{note:DimInt}
	Using Lemma~\ref{lemma:Intersection} we conclude that the dimension of the intersection of all the hyperplanes of the discriminantal arrangement $\mcl{C}^n_{\binom{n}{m+1}}$ is at least $m$. Also refer to Theorem $4$ in H.~H.~Crapo~\cite{MR0766269}.
\end{note}
Now we state the main theorem of this article.
\begin{thmA}[Main Theorem]
	\namedlabel{theorem:CountingInvariance}{A}
	~\\
	Let $n>m>1$ be two positive integers.
	Let $\mcl{N}=\{L_1,L_2,\ldots,L_n\}$ be a normal system of cardinality $n$ in $\mbb{R}^m$.
	Then 
	\begin{itemize}
		\item there is a bijection between the isomorphism classes of generic hyperplane arrangements $\mcl{H}_n^m$ with normal system $\mcl{N}$ 
		under isomorphisms which are trivial on subscripts, and the antipodal pairs of convex cones in the associated discriminantal arrangement $\mcl{C}^n_{\binom{n}{m+1}}$.
		\item Consequently, the number of such isomorphism classes of generic hyperplane arrangements $\mcl{H}_n^m$ is exactly equal to half of the number of convex cones in the 
		associated discriminantal arrangement $\mcl{C}^n_{\binom{n}{m+1}}$.
	\end{itemize}
\end{thmA}
\begin{remark}
For $m=2$ and $n=4,5$ the main theorem is discussed in detail in example Section~\ref{sec:Examples}. For $n=6,m=2,3$, it is discussed in detail in example Section~\ref{sec:NonConcurrencyFree}. These two sections help understand the statement of the main theorem better.
\end{remark}
\begin{remark}
\label{Remark:Methodology}
We first enumerate the isomorphism classes of generic hyperplane arrangements 
arising from a fixed normal system (refer to Definitions~[\ref{defn:NS},\linebreak~\ref{defn:HAGivenbyNS}, ~\ref{defn:NSAHA}]), under isomorphisms which preserve subscripts (refer to Definition~\ref{defn:Iso}).
This is done by computing the characteristic polynomial of the associated discriminantal arrangement (refer to Definition~\ref{defn:CA}, Note~\ref{note:Convention}) 
arising from a normal system. When the normal system is concurrency free (refer to Definition~\ref{defn:ConcurrencyFree}), the enumerated value is actually independent of the choice of such a normal system that we begin with and depends entirely on the combinatorics -- more precisely on the intersection lattice of the discriminantal arrangement.
This concurrency free restriction on the normal system is a mild restriction, and this condition is quite generic (refer to Theorem~\ref{theorem:generic} and Note~\ref{note:mildrestriction}).
\end{remark}

The method of computing characteristic polynomial for hyperplane arrangements is a well established method. 
Articles by T.~Zaslavky~\cite{MR0400066},~\cite{MR0357135}, F.~Ardila~\cite{MR2318445}, E.~Katz~\cite{MR3702317}, and books by 
A.~Dimca ~\cite{MR3618796}, P.~Orlik \& H.~Terao~\cite{MR1217488}, R.~Stanley~\cite{MR2868112} are relevant in which this concept is explained.

\section{\bf{Combinatorics of the discriminantal arrangement}}
\label{sec:CCA}
The combinatorics of the discriminantal arrangement -- more precisely the intersection lattice of the discriminantal arrangement -- has been studied in this section. The intersection lattice for a certain class of ``very generic or sufficiently general" (Definition $2.1$ in~\cite{MR1720104}) discriminantal arrangements which maximises the $f$-vector of the intersection lattice has been already characterised by C.~A.~Athanasiadis (refer to Theorem $2.3$ in~\cite{MR1720104}). An important ingredient in its proof is the Crapo's characterisation of the matroid $M(n,m)$ of circuits of the configuration of $n$-generic points in $\mbb{R}^m$. This matroid is introduced in H.~H.~Crapo~\cite{MR0766269} and characterised in H.~H.~Crapo~\cite{MR0843374}, Chapter 6, when the coordinates of the $n$-points are generic indeterminates, as the Dilworth completion $D_m(B_n)$ of the $m^{th}$-lower truncation of the Boolean algebra of rank $n$ (see H.~H.~Crapo and G.~C.~Rota~\cite{MR0290980}, Chapter 7). The intersection lattice of ``very generic" discriminantal arrangements coincides with the lattice $L(n,m)$ of flats of $M(n,m)$. In C.~A.~Athanasiadis~\cite{MR1720104}, it is proved that this lattice is isomorphic to the lattice $P(n,m)$ (refer to Theorem $3.2$ in~\cite{MR1720104}).  $P(n,m)$ is the collection of all sets of the form $\mcl{S}=\{S_1,S_2,\ldots,S_r\}$, where $S_i\subseteq\{1,2,\ldots,n\}$, each of cardinality at least $m+1$, such that 
\equ{\mid \us{i\in I}{\bigcup}S_i \mid>m+\us{i\in I}{\sum}(\mid S_i\mid-m)}
for all $I\subseteq \{1,2,\ldots,r\}$ with $\mid I\mid \geq 2$.  They partially order $P(n,m)$ by letting $\{S_1,S_2,\ldots,S_r\}=\mcl{S}\leq \mcl{T}=\{T_1,T_2,\ldots,T_p\}$, if, for each $1\leq i\leq r$ there exists $1\leq j\leq p$ such that $S_i\subseteq T_j$.  This isomorphism between the lattices $L(n,m)$ and $P(n,m)$ was initially conjectured (refer to Definition $4.2$ and Conjecture $4.3$) in M.~Bayer and K.~Brandt~\cite{MR1456579}. 

\begin{remark}
A normal system $\mcl{N}$ is said to be ``very generic", if it gives a ``very generic or sufficiently general" discriminantal arrangement in sense of \linebreak C.~A.~Athanasiadis (Definition $2.1$ in~\cite{MR1720104}) or ``very generic" discriminantal arrangement in the sense of M.~Bayer and K.~Brandt (Definition $4.2$  in~\cite{MR1456579}). Let $\mcl{N}=\{L_i=\{t(a_{i1},a_{i2},\ldots,a_{im})$ $\in \mbb{R}^m\mid t\in \R\}\mid 1\leq i\leq n\}$ be a fixed normal system which gives a ``very generic" discriminantal arrangement $\mcl{C}^n_{\binom{n}{m+1}}$ with coefficient matrix $[a_{ij}]_{1\leq i\leq n,1\leq i\leq m}$. Given an element $\mcl{S}=\{S_1,S_2,\ldots,S_r\}\in P(n,m)$, these sets $S_i,1\leq i\leq r$ can be thought of as sets of concurrencies of hyperplanes of an arrangement in $\mbb{R}^m$ given by the ``very generic" normal system $\mcl{N}$ using Theorem $3.2$ in~\cite{MR1720104}.  This idea helps understand this section on combinatorics better.  Also refer to Definition~\ref{defn:BaseConstruction} for an explanation of this idea.
\end{remark}

Though the main Theorem~\ref{theorem:CountingInvariance} is stated for any discriminantal arrangement, we show that the number of convex cones formed in a discriminantal arrangement $\mcl{C}^n_{\binom{n}{m+1}}$ does not depend on the normal system $\mcl{N}$, and is a combinatorial invariant if the normal system $\mcl{N}$ is concurrency free (refer to Definition~\ref{defn:ConcurrencyFree}). Actually we explore the relationship between an element in the intersection lattice and its combinatorial description to describe its rank precisely when the arrangement is concurrency free which later turns out to be a generic condition. This section gives a better geometric description of the rank similar to Corollary $3.6$ in~\cite{MR1720104}. We will revisit this Corollary $3.6$ in~\cite{MR1720104} once again in Definition~\ref{defn:BaseConstruction}.
We begin with a few combinatorial definitions after motivating the definitions with the following note.

\begin{note}
	With the notations of Definition~\ref{defn:CA} and convention in Note~\ref{note:Convention}, let $\mcl{D}$ be any collection of subsets of $\{1,2,\ldots,n\}$ each of size $m+1$. 
	We remark that the dimension of the intersection 
	\equ{\dim\bigg(\us{\{i_1,i_2,\ldots,i_m,i_{m+1}\} \in 
			\mcl{D}}{\bigcap} M_{\{i_1<i_2<\ldots<i_m<i_{m+1}\}}\bigg)}
	could possibly depend on the combinatorics of the sub-collection $\mcl{D}$ and not on the coefficients of the variables in the equations defining the hyperplanes 
	\equ{M_{\{i_1<i_2<\ldots<i_m<i_{m+1}\}}\text{ with  }\{i_1,i_2,\ldots,i_m,i_{m+1}\}\in \mcl{D}}
	if the normal system and the coefficients are of a certain type which is given in Definition~\ref{defn:ConcurrencyFree}.
	We prove in Theorem~\ref{theorem:DimIntersection} that this dimension is a combinatorial invariant only depending on the collection $\mcl{D}$ for normal systems which are
	concurrency free.
\end{note}
This note motivates the following definition.
\begin{defn}[Concurrency Closed Sub-collection and Concurrency Closure]
	\label{defn:CC}
	~\\
	Let $n>m$ be two positive integers. Let \equ{\mcl{E}=\{\{i_1,i_2,\ldots,i_{m+1}\}\mid 1\leq i_1<i_2<\ldots<i_m<i_{m+1}\leq n\}} be the collection of all subsets of 
	cardinality $m+1$. Let $\mcl{D}\subs \mcl{E}$ be any arbitrary collection. 
	
	We say $\mcl{D}$ is concurrency closed if the following criterion for any element $S\in \mcl{E}$ is satisfied with respect to $\mcl{D}$. Suppose there exists $\{S_1,S_2,\ldots,S_r\}\subseteq \mcl{D}$ and $S_i\neq S,1\leq i\leq r$ such that for every $J\subseteq \{1,2,\ldots,r\}$ we have $\mid \us{j\in J}{\cup}S_j\mid \geq m+\mid J\mid$ and $\mid \us{i=1}{\os{r}{\cup}}S_j\cup S\mid < m+r+1$ then $S\in \mcl{D}$. This definition is motivated by the notion of independence in the Dilworth matroid $D_m(B_n)$.
	
	We observe that the collection $\mcl{E}$ is concurrency closed.
    Now let $\mcl{D}\subs \mcl{E}$ be any arbitrary collection. Construct the concurrency closure $\overline{\mcl{D}}$ of $\mcl{D}$ as follows. 
	First set $\mcl{D}_0=\mcl{D}$ and add those elements $S\in \mcl{E}$ to $\mcl{D}_0$ if these $S$ satisfy the criterion mentioned above, to obtain $\mcl{D}_1$. Now construct $\mcl{D}_2$ from $\mcl{D}_1$ similarly and so on. 
	We have \equ{\mcl{D}_0=\mcl{D}\subsetneq \mcl{D}_1 \subsetneq \mcl{D}_2\subsetneq \ldots \subsetneq \mcl{D}_n=\overline{\mcl{D}}.}
	Since $\mcl{E}$ is a finite set we obtain $\overline{\mcl{D}}$ from $\mcl{D}_0$ in finitely many steps. Actually it can be shown that $\mcl{D}_1$ itself is concurrency closed and $\mcl{D}_1=\overline{\mcl{D}}$.
\end{defn}
\begin{defn}[Base Collection]
	\label{defn:BC}
	~\\
	Let $n>m$ be two positive integers. Let \equ{\mcl{E}=\{\{i_1,i_2,\ldots,i_{m+1}\}\mid 1\leq i_1<i_2<\ldots<i_m<i_{m+1}\leq n\}} be the collection of all subsets of 
	cardinality $m+1$. Let $\mcl{D}\subs \mcl{E}$ be any arbitrary collection. 
	We say $\ti{\mcl{D}}$ is a base collection for $\mcl{D}$ if $\overline{\ti{\mcl{D}}}=\overline{\mcl{D}}$ and $\ti{\mcl{D}}$ is minimal, that is, if 
	$\mcl{D}'$ is any other collection such that $\overline{\mcl{D}'}=\overline{\mcl{D}}$ and $\mcl{D}'\subseteq \ti{\mcl{D}}$ then we have $\ti{\mcl{D}}=\mcl{D}'$.  We can actually show that all minimal bases have equal cardinality.
\end{defn}
\begin{defn}[Construction of a Base for a Concurrency Closed Collection]
	\label{defn:BaseConstruction}
	Let $n>m$ be two positive integers. Let \equ{\mcl{E}=\{\{i_1,i_2,\ldots,i_{m+1}\}\mid 1\leq i_1<i_2<\ldots<i_m<i_{m+1}\leq n\}} be the collection of all subsets of 
	cardinality $m+1$. Let $\mcl{D}\subs \mcl{E}$ be a concurrency closed subcollection. We say there is a concurrency of order $k\geq m+1$ in $\mcl{D}$, if there exists a concurrency set $D\subs \{1,2,\ldots,n\}$
	of size $k$ such that all $\binom{k}{m+1}$ subsets of $D$ of size $m+1$ are in the collection $\mcl{D}$. Moreover $D$ should be maximal with respect to this property, that is, there does not exist a set $E \supsetneq D$
	of size more than $k$ such that all $\binom{\mid E \mid}{m+1}$ subsets of size $m+1$ are in the collection $\mcl{D}$.
	Let $k_1,k_2,\ldots,k_r$ be the orders of concurrencies that exist in $\mcl{D}$ with $k_i\geq m+1, 1\leq i\leq r$. Then the cardinality of a base collection $\mcl{D}'$ for $\mcl{D}$ is given by 
	\equ{(k_1-m)+(k_2-m)+\ldots+(k_r-m)=\bigg(\us{i=1}{\os{r}{\sum}}k_i\bigg)-rm.}
	Also see Corollary $3.6$ in C.~A.~Athanasiadis~\cite{MR1720104}.  Let $D_i\subs \{1,2,\ldots,n\}$ be the concurrency set of size $k_i$ which gives rise to the order $k_i$ concurrency in the concurrency closed subcollection $\mcl{D}$. Let $\mcl{S}=\{D_1,D_2,\ldots,D_r\}$. If the normal system is concurrency free (Definition~\ref{defn:ConcurrencyFree}) then we have that 
	\equa{\#\big(\mcl{D}'&=Base\ of\ (\mcl{D})\big) = n-\dim\bigg(\us{\{i_1,i_2,\ldots,i_m,i_{m+1}\} \in \mcl{D}}{\bigcap} M_{\{i_1<i_2<\ldots<i_m<i_{m+1}\}}\bigg)\\&=\us{D\in \mcl{S}}{\sum}\gn(D)} 
	in the notation of Corollary 3.6 in~\cite{MR1720104} where $\gn(D)=\max(0,\mid D\mid-\ m)$ for $D\subs\{1,2,\ldots,n\}$. We also can show in this case that $\mcl{S}\in P(n,m)$. Here we do something more. We actually construct a base collection $\mcl{D}'$ for $\mcl{D}$.
	Let the concurrency sets be given by 
	\equ{D_i=\{j^i_1<j^i_2<\ldots<j^i_{k_i}\}, 1\leq i\leq r.}
	Then a base collection $\mcl{D}'$ for $\mcl{D}$ is given by 
	\equ{\{\{j^i_1,j^i_2,\ldots,j^i_m,j^i_l\}\mid m+1 \leq l \leq k_i,1\leq i\leq r\}.}
	This collection $\mcl{D}'$ has the required cardinality. We denote \equ{rank(\mcl{D})=\us{D\in \mcl{S}}{\sum}\gn(D) \ \ (\text{is the cardinality of any base collection of } \mcl{D}).}
\end{defn}
Now we mention a note on line arrangements which motivates the definition of a normal system being concurrency free.

\begin{note}
	\label{note:Nonconcurrencyfree}
	Let $L_1,L_2,L_3,L_4,L_5,L_6$ be six lines in the plane $\mbb{R}^2$ with no two of them parallel and such that \equ{L_1\perp L_4,L_2\perp L_5,L_3\perp L_6.} Suppose the sets  
	\equ{\{L_2,L_3,L_4\},\{L_1,L_2,L_6\},\{L_1,L_3,L_5\}} of lines are concurrent. Then the fourth set $\{L_4,L_5,L_6\}$ of lines is 
	also concurrent as the altitudes in the triangle $\Gd L_1L_2L_3$ must be concurrent. Here we have $n=6,m=2, \binom{n}{m+1}=\binom{6}{3}=20$. 
	Let 
	\equa{&\mcl{C}^6_{20}=\\
		&\{M_{\{1,2,3\}},M_{\{1,2,4\}},M_{\{1,2,5\}},M_{\{1,2,6\}},M_{\{1,3,4\}},M_{\{1,3,5\}},M_{\{1,3,6\}},M_{\{1,4,5\}},\\
		&M_{\{1,4,6\}},M_{\{1,5,6\}},M_{\{2,3,4\}},M_{\{2,3,5\}},M_{\{2,3,6\}},M_{\{2,4,5\}},M_{\{2,4,6\}},M_{\{2,5,6\}},\\
		&M_{\{3,4,5\}},M_{\{3,4,6\}},M_{\{3,5,6\}},M_{\{4,5,6\}}\}}
	be the associated discriminantal arrangement. Let $\mcl{E}$ be the collection all subsets of $\{1,2,3,4,5,6\}$ of cardinality three.
	Then we observe that the sub-collections of $\mcl{E}$ given by
	\equ{\mcl{D}_1=\{\{1,2,6\},\{1,3,5\},\{2,3,4\}\}\text{ and }\mcl{D}_2=\mcl{D}_1\cup \{\{4,5,6\}\}} 
	are concurrency closed, that is, $\overline{\mcl{D}_1}=\mcl{D}_1$ and $\overline{\mcl{D}_2}=\mcl{E}$. The concurrency orders that exist in $\overline{\mcl{D}_1}$ are $3,3,3$ and in $\overline{\mcl{D}_2}$ are $6$. A base collection
	for $\overline{\mcl{D}_1}$ is $\mcl{D}_1$ and a base collection for $\overline{\mcl{D}_2}$ is $\mcl{D}_2$.
	However we also observe that because of perpendicularity of the pairs of lines $L_1\perp L_4,L_2\perp L_5,L_3\perp L_6$ the following spaces
	\equ{\us{\{i_1<i_2<i_3\}\in \mcl{D}_1}{\bigcap}M_{\{i_1<i_2<i_3\}}=\us{\{i_1<i_2<i_3\}\in \mcl{D}_2}{\bigcap}M_{\{i_1<i_2<i_3\}}}
	are equal because altitudes of $\Gd L_1L_2L_3$ must be concurrent. Hence the dimensions are equal but the cardinality of their base collections are different.  Moreover we have that 
	\equ{\us{\{i_1<i_2<i_3\}\in \mcl{D}_2}{\bigcap}M_{\{i_1<i_2<i_3\}}\neq \us{\{i_1<i_2<i_3\}\in \overline{\mcl{D}_2}}{\bigcap}M_{\{i_1<i_2<i_3\}}}
\end{note}
This motivates the following Definition~\ref{defn:ConcurrencyFree} in which we define when a normal system is concurrency free.
\begin{defn}[Concurrency Free Normal System]
	\label{defn:ConcurrencyFree}
	~\\
	Let $n>m>1$ be two positive integers.
	Let $\mcl{N}=\{L_1,L_2,\ldots,L_n\}$ be a normal system of cardinality $n$ in $\mbb{R}^m$. Let \equ{\mcl{U}=\{\pm (a_{i1},a_{i2},\ldots,a_{im})\mid (a_{i1},a_{i2},\ldots,a_{im})\in L_i, 1\leq i\leq n\}} 
	be a set of antipodal pairs of vectors of the normal system $\mcl{N}$. We fix the coefficient matrix $[a_{ij}]_{1\leq i\leq n,1\leq j\leq m}\in M_{n\times m}(\mbb{R})$.
	Let $\mcl{C}^n_{\binom{n}{m+1}}=\{M_{\{i_1,i_2,\ldots,i_{m+1}\}}\mid 1\leq i_1<i_2<\ldots<i_{m+1}\leq n\}$ be the associated discriminantal arrangement. 
	Let $\mcl{E}$ be the collection of all subsets of $\{1,2,\ldots,n\}$ of size $m+1$. 
	We say the normal system $\mcl{N}$ is concurrency free, if for any concurrency closed collection $\mcl{D}\subs \mcl{E}$ we have
	\equa{\#\big(\mcl{D}'=Base\ of\ (\mcl{D})\big) &= n-\dim\bigg(\us{\{i_1,i_2,\ldots,i_m,i_{m+1}\} \in \mcl{D}}{\bigcap} M_{\{i_1<i_2<\ldots<i_m<i_{m+1}\}}\bigg)\\
	&=n-\dim\bigg(\us{\{i_1,i_2,\ldots,i_m,i_{m+1}\} \in \mcl{D}'}{\bigcap} M_{\{i_1<i_2<\ldots<i_m<i_{m+1}\}}\bigg)}
	
	The example in Note~\ref{note:Nonconcurrencyfree} gives a normal system consisting of cardinality six in the plane which is not concurrency free.
\end{defn}
\begin{note}[Breaking concurrency orders in succession]
	\label{note:BreakConcurrency}
	Let $\mcl{N}$ be a normal system which is concurrency free. Given a finite set of hyperplanes with normals along the lines of the normal system $\mcl{N}$, let $\mcl{D}$ be the concurrency closed
	collection generated by the higher order concurrencies. Then the point concurrencies of higher orders can be broken by translations of the hyperplanes in succession exactly $b$-times where 
	$b$ is the cardinality of a base $\mcl{D}'$ for $\mcl{D}$. This is because normals of the hyperplanes corresponding to sets in the base collection are linearly independent. For an example of breaking concurrency orders in succession see Example~\ref{example:FiveLA} and for a theoretical explanation refer to Section~\ref{sec:gennature}.
\end{note}
\begin{note}
	\label{note:ConcurrencyFree}
	That any normal system $\mcl{N}$ which consists of four or five lines in the plane $\mbb{R}^2$ is concurrency free follows from Examples~[\ref{example:FourLA},~\ref{example:FiveLA}].
	So a normal system needs to have at least six lines as an example which is not concurrency free.
\end{note}
Now we state a theorem which gives the dimension as a combinatorial invariant.
\begin{theorem}[Dimension of the Intersection]
	\label{theorem:DimIntersection}
	~\\
	Let $n>m>1$ be two positive integers.
	Let \equ{\mcl{C}^n_{\binom{n}{m+1}}=\{M_{\{i_1<i_2<\ldots<i_m<i_{m+1}\}}\mid 1\leq i_1<i_2<\ldots<i_m<i_{m+1}\leq n\}} 
	be the discriminantal arrangement given by a normal system $\mcl{N}=\{L_1,L_2,\ldots,L_n\}$ in $\mbb{R}^m$ which is concurrency free. Let 
	\equ{\mcl{E}=\{\{i_1,i_2,\ldots,i_m,i_{m+1}\}\mid 1\leq i_1<i_2<\ldots<i_m<i_{m+1}\leq n\}.}
	Let $\mcl{D}\subs \mcl{E}$ be a sub-collection of sets. Then we have   
	\equa{d_{\mcl{D}}=n-rank(\overline{\mcl{D}})&=\dim\bigg(\us{\{i_1,i_2,\ldots,i_m,i_{m+1}\} \in 
			\mcl{D}}{\bigcap} M_{\{i_1<i_2<\ldots<i_m<i_{m+1}\}}\bigg)}
	and $d_{\mcl{D}}$ is a combinatorial invariant.
\end{theorem}
\begin{proof}
	By definition of the concurrency closure and the base collection we have 
	\equa{&\us{\{i_1,i_2,\ldots,i_m,i_{m+1}\} \in \mcl{D}}{\bigcap} M_{\{i_1<i_2<\ldots<i_m<i_{m+1}\}}=\\ 
		&\us{\{i_1,i_2,\ldots,i_m,i_{m+1}\} \in \overline{\mcl{D}}=\overline{\ti{\mcl{D}}}}{\bigcap} M_{\{i_1<i_2<\ldots<i_m<i_{m+1}\}}=\\
		&\us{\{i_1,i_2,\ldots,i_m,i_{m+1}\} \in \ti{\mcl{D}}}{\bigcap} M_{\{i_1<i_2<\ldots<i_m<i_{m+1}\}}.}
	Hence their dimensions are equal. For the base collection $\ti{\mcl{D}}$ the set of normals of the hyperplanes 
	\equ{M_{\{i_1,i_2,\ldots,i_m,i_{m+1}\}},\{i_1,i_2,\ldots,i_m,i_{m+1}\} \in \ti{\mcl{D}}}
	in $\mbb{R}^n$ are linearly independent because of minimality and the fact that the normal system $\mcl{N}$ is concurrency free. 
	This proves the theorem.
\end{proof}
We compute the characteristic polynomial of the discriminantal arrangement over $\mbb{R}$ given by a normal system over $\mbb{R}$ which is concurrency free.

\begin{theorem}
	\label{theorem:CharpolyCA}
	Let $n>m>1$ be two positive integers.
	Let \equ{\mcl{C}^n_{\binom{n}{m+1}}=\{M_{\{i_1<i_2<\ldots<i_m<i_{m+1}\}}\mid 1\leq i_1<i_2<\ldots<i_m<i_{m+1}\leq n\}} 
	be the discriminantal arrangement given by a normal system $\mcl{N}=\{L_1,L_2,\ldots,L_n\}$ which is concurrency free. 
	Let \equ{\mcl{E}=\{\{i_1,i_2,\ldots,i_m,i_{m+1}\}\mid 1\leq i_1<i_2<\ldots<i_m<i_{m+1}\leq n\}.}
	Then the characteristic polynomial is given by
	\equ{\gch(\mcl{C}^n_{\binom{n}{m+1}})(x)
		=\us{\mcl{D}\subs \mcl{E}}{\sum}(-1)^{\#(\mcl{D})}x^{d_{\mcl{D}}}}
	with $d_{\mcl{D}}=n-rank(\overline{\mcl{D}})$ where $rank(\overline{\mcl{D}})$ is the cardinality of a base collection $\ti{\mcl{D}}$ for $\overline{\mcl{D}}$.
	The number of convex regions is given by 
	\equa{r(\mcl{C}^n_{\binom{n}{m+1}})&=(-1)^n\gch(\mcl{C}^n_{\binom{n}{m+1}})(-1)=(-1)^n\us{\mcl{D}\subs \mcl{E}}{\sum}(-1)^{\#(\mcl{D})+d_{\mcl{D}}}\\
		&=\us{\mcl{D}\subs \mcl{E}}{\sum}(-1)^{n+\#(\mcl{D})+d_{\mcl{D}}}=\us{\mcl{D}\subs \mcl{E}}{\sum}(-1)^{\#(\mcl{D})+rank(\overline{\mcl{D}})}.}
	Both are independent of any normal system which is concurrency free and depends only on the cardinalities $n\geq m+1$ and is combinatorially determined.
\end{theorem}
\begin{proof}
	Let $\mcl{L}$ be the set of all intersections of hyperplanes in the discriminantal arrangement $\mcl{C}^n_{\binom{n}{m+1}}$.
	The characteristic polynomial is given by 
	\equ{\gch(\mcl{C}^n_{\binom{n}{m+1}})(x)=\us{\mcl{D}\subs\mcl{E}}{\sum}(-1)^{\#(\mcl{D})}x^{n-rank(\overline{\mcl{D}})},}
	(refer to R.~Stanley~\cite{MR2868112}, page $283$, Proposition $3.11.3$). Using the dimension numbers $d_{\mcl{D}}$ for elements of $\mcl{L}$ we can determine the 
	characteristic polynomial of the discriminantal arrangement as  
	\equ{\gch(\mcl{C}^n_{\binom{n}{m+1}})(x)=\us{\mcl{D}\subs \mcl{E}}{\sum}(-1)^{\#(\mcl{D})}q^{d_{\mcl{D}}}=x^n-\binom{n}{m+1}x^{n-1}+\ldots.}
	
	Hence the number of convex regions of the discriminantal arrangement $\mcl{C}^n_{\binom{n}{m+1}}$ is given by $(-1)^n\gch(\mcl{C}^n_{\binom{n}{m+1}})(-1)$ and the theorem follows.
\end{proof}
Next we mention a few properties of the characteristic polynomial.
\begin{note}[Properties of the Characteristic Polynomial of the Discriminantal Arrangement]
	~\\
	\begin{itemize}
		\item The characteristic polynomial which arises from a discriminantal arrangement will have coefficients alternating in signs.
		\item $x^{m}(x-1)$ is a factor of $\gch(\mcl{C}^n_{\binom{n}{m+1}})(x)$ with 
		$\gch(\mcl{C}^n_{\binom{n}{m+1}})(1)=0$. 
		\item There are no bounded convex regions for the discriminantal arrangement and all its regions are unbounded.
	\end{itemize}
\end{note}
\section{\bf{Examples of four-line and five-line arrangements in the plane}}
\label{sec:Examples}
Here in this section we mention two examples of four-line and five-line arrangements. Later we make some observations which motivate the statement
of the main Theorem~\ref{theorem:CountingInvariance}.
\begin{example}
	\label{example:FourLA}
	Let $\mcl{N}=\{N_1,N_2,N_3,N_4\}$ be four lines in the plane $\mbb{R}^2$ giving rise to a normal system in $\mbb{R}^2$. 
	Then this gives rise to a discriminantal arrangement
	\equ{\mcl{C}^4_4=\{M_{\{1<2<3\}},M_{\{1<2<4\}},M_{\{1<3<4\}},M_{\{2<3<4\}}\}} in $\mbb{R}^4$. Let $\mcl{L}_4^2=\{L_1,L_2,L_3,L_4\}$ be any generic line arrangement consisting of four lines in 
	$\mbb{R}^2$ given by the normal system $\mcl{N}$. On any line $L_i$, there are exactly three vertices $L_i\cap L_j, L_i\cap L_k, L_i\cap L_l$ for $j,k,l\in\{1,2,3,4\}\bs \{i\}$. The middle vertex out of these three vertices is called the central point of the line $L_i$. There are exactly two lines (say) $L_i,L_j$ among $L_k,1\leq k\leq 4$ whose central points concide. This common central point must therefore be the intersection point of the two lines $L_i, L_j$. We define this common central point $P=L_i\cap L_j$ to be the	special point of the generic four-line arrangement $\mcl{L}_4^2$ as in Figure~\ref{fig:MinusOne}. It is unique for a generic four-line arrangement.
	\begin{figure}[h]
		\centering
		\includegraphics[width = 0.6\textwidth]{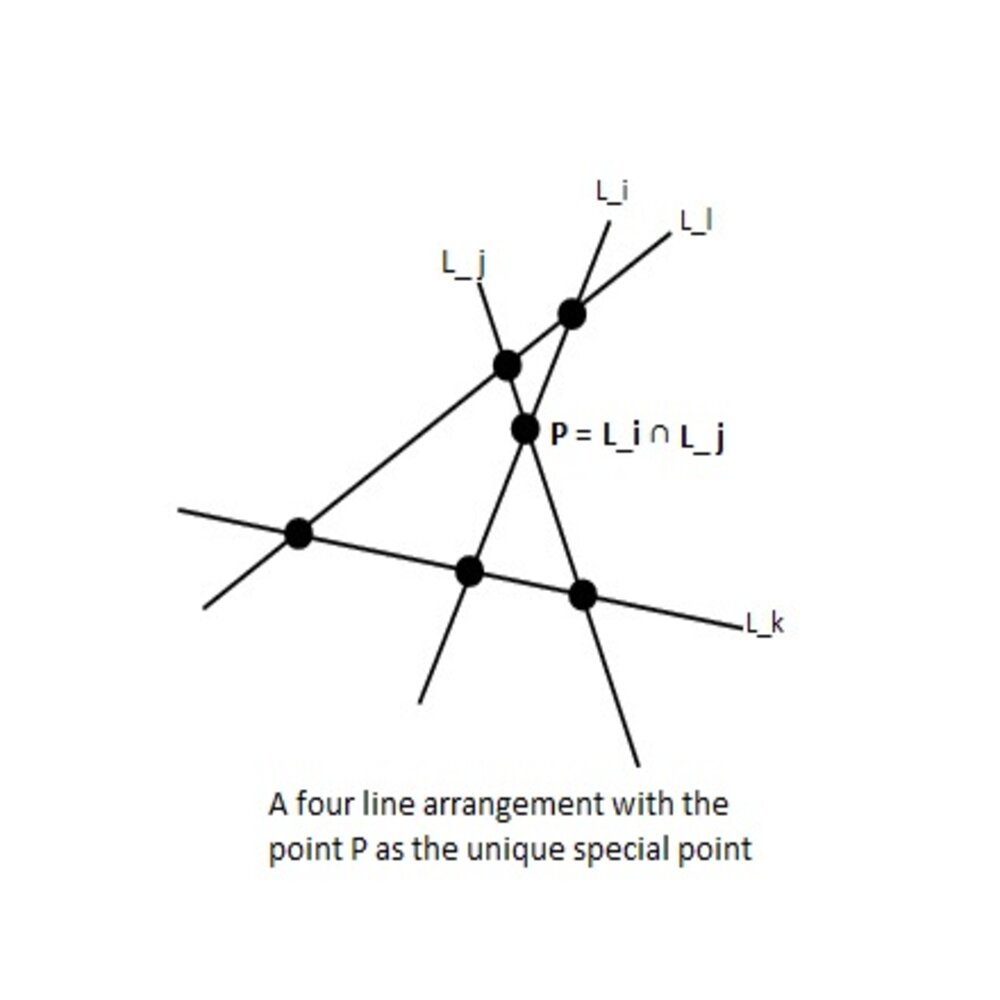}
		\caption{A Generic Four-Line Arrangement with Unique Special Point $P=L_i\cap L_j$}
		\label{fig:MinusOne}
	\end{figure}
	 If the subscripts $1,2,3,4$ of the lines are numbered with increasing order of angles the lines make, with respect to $X$-axis in the plane then there
	are four possibilities for special points of the four-line arrangements given by 
	\equ{L_1\cap L_2, L_2\cap L_3, L_3\cap L_4,L_4\cap L_1}
	as shown in Figure~\ref{fig:Zero} and the other two possibilities $L_1\cap L_3, L_2\cap L_4$ do not occur. Also we have two triangular regions in any generic four-line arrangement. They are given as follows.
	\begin{center}
		\begin{tabular}{ | l | l |}
			\hline
			Special Point & Triangular Regions \\ \hline
			$(12)=L_1\cap L_2$ & $(123)=\Gd L_1L_2L_3,(124)=\Gd L_1L_2L_4$\\ \hline
			$(23)=L_2\cap L_3$ & $(123)=\Gd L_1L_2L_3,(234)=\Gd L_2L_3L_4$\\ \hline
			$(34)=L_3\cap L_4$ & $(134)=\Gd L_1L_3L_4,(234)=\Gd L_2L_3L_4$\\ \hline
			$(14)=L_1\cap L_4$ & $(124)=\Gd L_1L_2L_4,(134)=\Gd L_1L_3L_4$\\ \hline
		\end{tabular}
	\end{center}
	\begin{figure}[h]
	\centering
	\includegraphics[width = 0.9\textwidth]{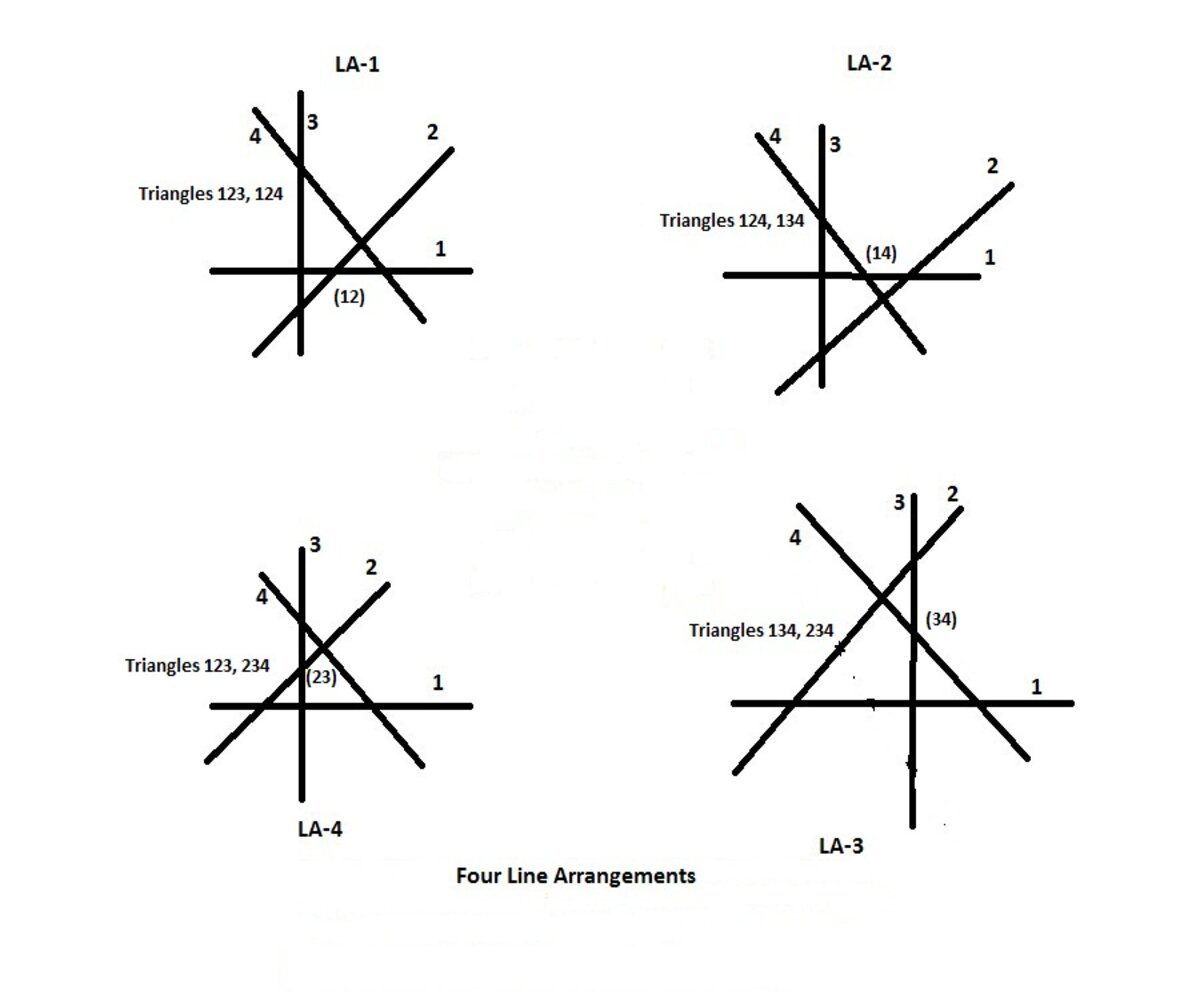}
	\caption{Isomorphism Classes of Generic Four-Line Arrangements in the Plane with Special Points LA-1: (12),LA-4: (23),LA-3: (34),LA-2: (14)}
	\label{fig:Zero}
\end{figure}
	These give rise to four different isomorphism classes as there cannot be an isomorphism between any two among them preserving the subscripts.
	We observe that each conical convex region of the discriminantal arrangement $\mcl{C}^4_4$ is bounded by two hyperplanes as there are two triangular regions for any generic four-line arrangement. This becomes clear in the following.
	Using Note~\ref{note:DimInt} the common intersection of hyperplanes of the discriminantal arrangement $\mcl{C}^4_4$ in $\mbb{R}^4$ is at least two-dimensional and by distinctness of the four three-dimensional hyperplanes
	it is exactly two-dimensional. We can project this two-dimensional subspace of $\mbb{R}^4$ to a point yielding a projection of $\mbb{R}^4$ to $\mbb{R}^2$. 
	Now the four distinct three-dimensional hyperplanes correspond under projection to four distinct lines passing through the origin 
	giving rise to eight conical convex regions for the discriminantal arrangement $\mcl{C}^4_4$. The opposite conical regions give the same isomorphism class under preservation of subscripts.
	So we obtain in another way that there are exactly four isomorphism classes of four-line arrangements under isomorphisms which preserve subscripts. 
	
	Let $\mcl{E}=\{\{1,2,3\},\{1,2,4\},\{1,3,4\},\{2,3,4\}\}$. There are five concurrency closed sub-collections are \equ{\{\{1,2,3\}\},\{\{1,2,4\}\},\{\{1,3,4\}\},\{\{2,3,4\}\},\mcl{E}.}
	Hence the normal system $\mcl{N}$ is concurrency free (refer to Definition~\ref{defn:ConcurrencyFree}). The characteristic polynomial of a discriminantal arrangement $\mcl{C}^4_4$
	is given by 
	\equa{&\gch(\mcl{C}^4_4)(x)=x^4-4x^3+3x^2=x^2(x-1)(x-3)\\ 
		&\text{ with }r(\mcl{C}^4_4)=(-1)^4\gch(\mcl{C}^4_4)(-1)=8.}
	There is only one isomorphism class under a general isomorphism (which need not preserve subscripts) of generic line arrangements consisting
	of four lines.
\end{example}
We present an example of five lines in the plane.
\begin{example}
	\label{example:FiveLA}
	Let $\mcl{N}=\{N_1,N_2,N_3,N_4,N_5\}$ be five lines in the plane $\mbb{R}^2$ giving rise to a normal system. 
	Then this gives rise to a discriminantal arrangement
	\equa{\mcl{C}^5_{10}&=\{M_{\{1<2<3\}},M_{\{1<2<4\}},M_{\{1<2<5\}},M_{\{1<3<4\}},M_{\{1<3<5\}},M_{\{1<4<5\}},\\
		&M_{\{2<3<4\}},M_{\{2<3<5\}},M_{\{2<4<5\}},M_{\{3<4<5\}}\}} 
	in $\mbb{R}^5$.
	Let $\mcl{E}=\{\{i_1<i_2<i_3\}\mid 1\leq i_1<i_2<i_3\leq 5\}$. The concurrency closed sub-collections are given by
	\begin{enumerate}
		\item $\{\{i_1<i_2<i_3\}\}, 1\leq i_1<i_2<i_3\leq 5$.
		\item $\{\{i_1<i_2<i_3\},\{j_1<j_2<j_3\}\}, 1\leq i_1<i_2<i_3\leq 5,1\leq j_1<j_2<j_3\leq 5,\{i_1,i_2,i_3\}\cup\{j_1,j_2,j_3\}=\{1,2,3,4,5\}$.
		\item $\{\{i_1<i_2<i_3\},\{i_1<i_2<i_4\},\{i_1<i_3<i_4\},\{i_2<i_3<i_4\}\}, 1\leq i_1<i_2<i_3<i_4\leq 5$.
		\item $\mcl{E}$.
	\end{enumerate}
	We observe that for any five-line arrangement $\mcl{L}^2_{5}$ in the plane $\mbb{R}^2$ the concurrencies of higher order corresponding to the above sub-collections $(1),(2),(3),(4)$ are of the following types.
	\begin{enumerate}
		\item Single concurrency of order $3$ which can be broken.
		\item Two concurrencies of order $3$ out of which exactly one line is common. This type can be broken into the previous type by a translation of one non-common line.
		\item One single concurrency of order $4$ which can be broken.
		\item One single concurrency of order $5$ which can also be broken. 
	\end{enumerate}
	Also refer to Note~\ref{note:BreakConcurrency}. Hence we have the dimensions of the intersections given by 
	\begin{enumerate}
		\item $\dim(M_{\{i_1<i_2<i_3\}} \cap M_{\{j_1<j_2<j_3\}})=3$ if $\#\big(\{i_1,i_2,i_3\}\cup\{j_1,j_2,j_3\}\big)=4\text{ or }5$.
		\item $\dim(M_{\{i_1<i_2<i_3\}} \cap M_{\{j_1<j_2<j_3\}} \cap M_{\{k_1<k_2<k_3\}})=3$\\ if $\#\big(\{i_1,i_2,i_3\}\cup \{j_1,j_2,j_3\}\cup\{k_1,k_2,k_3\}\big)=4$.
		\item $\dim(M_{\{i_1<i_2<i_3}\} \cap M_{\{j_1<j_2<j_3\}} \cap M_{\{k_1<k_2<k_3\}})=2$\\ if $\#\big(\{i_1,i_2,i_3\}\cup \{j_1,j_2,j_3\}\cup\{k_1,k_2,k_3\}\big)=5$.
		\item $\dim(M_{\{i_1<i_2<i_3}\} \cap M_{\{j_1<j_2<j_3\}} \cap M_{\{k_1<k_2<k_3\}} \cap M_{\{l_1<l_2<l_3\}})=3$\\ if $\#\big(\{i_1,i_2,i_3\}\cup \{j_1,j_2,j_3\}\cup\{k_1,k_2,k_3\}\cup \{l_1,l_2,l_3\}\big)=4$.
		\item $\dim(M_{\{i_1<i_2<i_3}\} \cap M_{\{j_1<j_2<j_3\}} \cap M_{\{k_1<k_2<k_3\}} \cap M_{\{l_1<l_2<l_3\}})=2$\\ if $\#\big(\{i_1,i_2,i_3\}\cup \{j_1,j_2,j_3\}\cup\{k_1,k_2,k_3\}\cup \{l_1,l_2,l_3\}\big)=5$.
	\end{enumerate}
	The higher order intersections have dimension $2$. Hence any normal system $\mcl{N}$ of lines in plane $\mbb{R}^2$ of cardinality five is concurrency free.
	
	To count the number of convex regions we compute the characteristic polynomial of the discriminantal arrangement $\mcl{C}^5_{10}$ arising
	from a normal system. This polynomial is given by 
	\equa{\gch(\mcl{C}^5_{10})(x)&=x^5-[(10x^4)-(15x^3+30x^3)+(20x^3+100x^2)-(5x^3+205x^2)\\
		&+252x^2-210x^2+120x^2-45x^2+10x^2-x^2]\\
		&=x^5-10x^4+30x^3-21x^2.}
	So the characteristic polynomial which is alternating in signs is given by
	\equ{\gch(\mcl{C}^5_{10})(x)=x^5-10x^4+30x^3-21x^2=x^2(x-1)(x^2-9x+21).}
	Hence the number of convex regions is given by 
	\equ{r(\mcl{C}^5_{10})=(-1)^5\gch(\mcl{C}^5_{10})(-1)=62.}
	We know that opposite cones of the discriminantal arrangement $\mcl{C}^5_{10}$ correspond to isomorphic generic line arrangements under isomorphisms which preserve subscripts.
	Now we index the subscripts $1,2,3,4,5$ of the lines with increasing order of the angles the lines make with respect to the $X$-axis and list out $31=\frac{r(\mcl{C}^5_{10})}{2}$ 
	distinct isomorphism classes of line arrangements. These are the only and all possibilities..
	First we list $7$ five-line arrangements in the plane as shown in Figure~\ref{fig:One}.
	\begin{figure}[h]
		\centering
		\includegraphics[width = 1.0\textwidth]{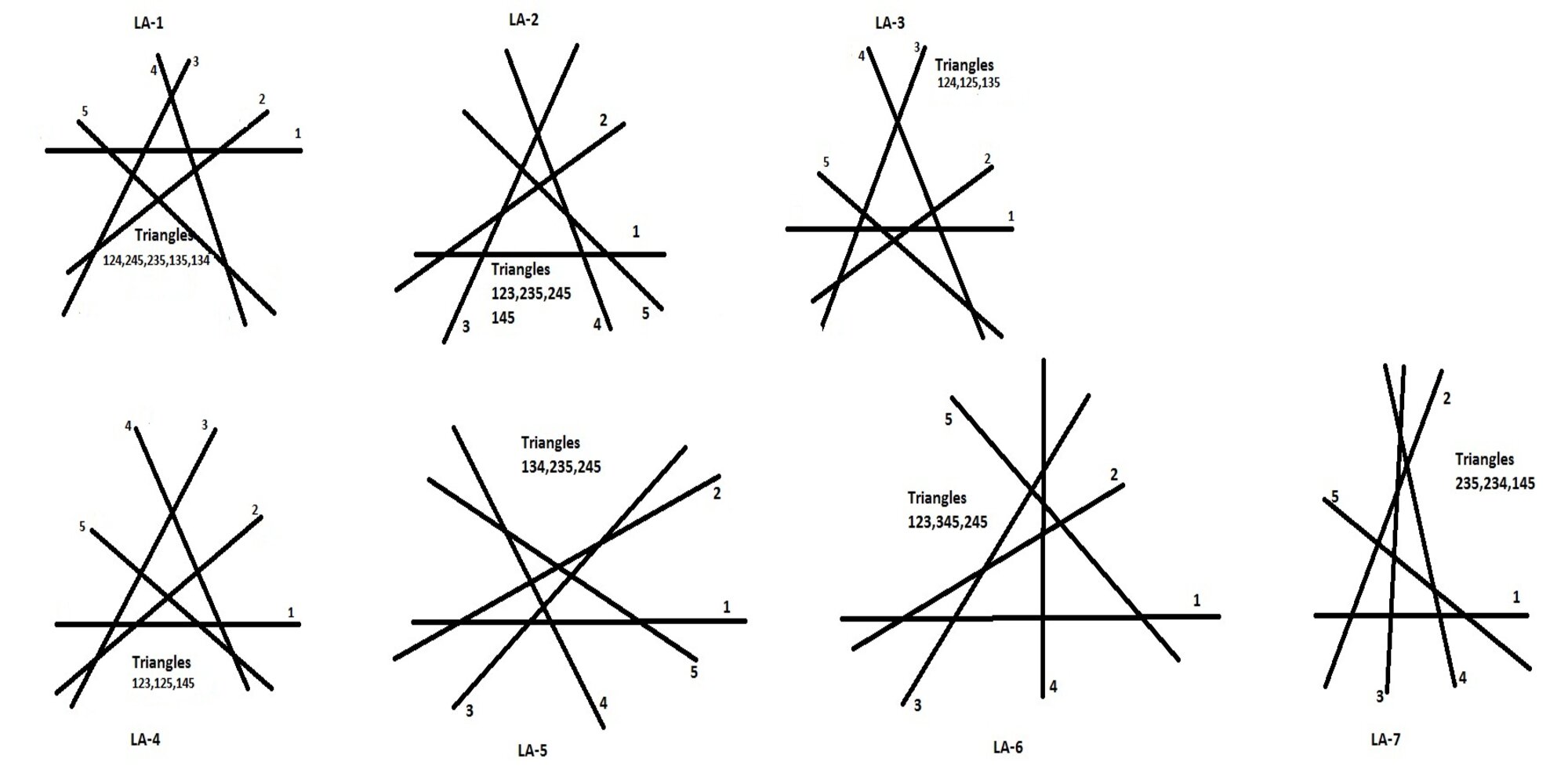}
		\caption{Seven Five-Line Arrangements in the Plane}
		\label{fig:One}
	\end{figure}
	The set of triangles for the first one
	is invariant under cyclic translates because of its symmetry and there are five cyclic translates for each of the 
	remaining ones:
	\begin{enumerate}
		\item \{124,245,235,135,134\}
		\item \equa{\{123,235,245,145\} \lra &\{234,134,135,125\},\{345,245,124,123\},\\ &\{145,135,235,234\},\{125,124,134,345\},} 
		\item \equa{\{135,125,124\} \lra &\{124,123,235\},\{235,234,134\},\{134,345,245\},\\ &\{245,145,135\},}
		\item \equa{\{123,125,145\} \lra &\{234,123,125\},\{345,234,123\},\{145,345,234\},\\ &\{125,145,345\},}
		\item \equa{\{134,235,245\} \lra &\{245,134,135\},\{135,245,124\},\{124,135,235\},\\ &\{235,124,134\},}
		\item \equa{\{345,123,245\} \lra &\{145,234,135\},\{125,345,124\},\{123,145,235\},\\ &\{234,125,134\},}
		\item \equa{\{235,234,145\} \lra &\{134,345,125\},\{245,145,123\},\{135,125,234\},\\ &\{124,123,345\},}
	\end{enumerate}
	totalling to $31$ distinct cones as they all have different sets of triangles. Hence there are exactly
	\equa{&\text{thirty one isomorphism classes under isomorphisms}\\&\text{which are trivial on the subscripts. }} 
\end{example}
\begin{note}
	Not all normal systems of lines in the plane of cardinality six are concurrency free (refer to Note~\ref{note:Nonconcurrencyfree}) unlike
	normal systems of cardinality four and five in the plane in Examples~[\ref{example:FourLA},~\ref{example:FiveLA}]. 
\end{note}
\begin{note}
	Examples~[\ref{example:FourLA},~\ref{example:FiveLA}] motivate the following question.
	For any normal system $\mcl{N}=\{L_1,L_2,\ldots,L_n\}$ in $\mbb{R}^m, n>m>1$, which cones of its associated discriminantal arrangement $\mcl{C}^n_{\binom{n}{m+1}}$ 
	correspond to the same isomorphism class under isomorphisms which preserve subscripts? We answer this question in the next section.
\end{note}
\section{\bf{The main result}}
\label{sec:MainResult}
~\\
The number of cones of the discriminantal arrangement $\mcl{C}^n_{\binom{n}{m+1}}$ and the homotopy type of the discriminantal arrangement have been studied before. Though generically the homotopy type does not change, the type does depend on the normal system that we begin with.  Example $3.2$ in M.~Falk~\cite{MR1209098} gives a discriminantal arrangement whose homotopy type is different from that of the generic type of ``very generic" discriminantal arrangements. This is further discussed below in Example~\ref{Example:SixPlanes}. We will give another such Example~\ref{Example:SixLines} in Section~\ref{sec:NonConcurrencyFree} which is more relevant to this article. In a much recent article A.~Libgober and S.~Settepanella~\cite{MR3899551} describe discriminantal arrangements which admit a codimension two strata of multiplicity three, whereas a generic type of ``very generic" discriminantal arrangements in the sense of C.~A.~Athanasiadis~\cite{MR1720104} has the property that they admit a codimension two strata only having multiplicity two or $m+2$. 

In this section we prove a few preliminary results before we prove main Theorem~\ref{theorem:CountingInvariance}.  
\begin{theorem}
	\label{theorem:CA}
	~\\
	Let $n>m>1$ be two positive integers.
	Let $\mcl{N}=\{L_1,L_2,\ldots,L_n\}$ be a normal system in $\mbb{R}^m$. Let $\mcl{C}^n_{\binom{n}{m+1}}$ be its associated discriminantal arrangement. Let $\mcl{H}_n^m=\{H_1,H_2,\ldots,H_n\}$ be an arrangement given by the normal system. Let the constant coefficient vector 
	be given by $(b_1,b_2,\ldots,b_n)\in \mbb{R}^n$ which lies in the interior of a cone $C$ of the discriminantal arrangement $\mcl{C}^n_{\binom{n}{m+1}}$.
	Suppose the subscripts $1\leq i_1<i_2<\ldots<i_m<i_{m+1}\leq n$ give rise to an $m$-dimensional 
	simplex polyhedrality, a polyhedral region in the arrangement $\mcl{H}_n^m$. Let the constant coefficient vector $(b_1,b_2,\ldots,b_n)$ 
	which lies in the interior of the cone $C$ be moved to the interior of its adjacent cone $D$ (say) by passing through single boundary hyperplane (co-dimension one) corresponding to $1\leq i_1<i_2<\ldots<i_m<i_{m+1}\leq n$ to a new constant coefficient vector $(\ti{b}_1,\ti{b}_2,\ldots,\ti{b}_n)$ giving rise to a new hyperplane arrangement $\ti{\mcl{H}}_n^m$. Let $A=\{j_1,j_2,\ldots,j_{m-1}\} \subs \{i_1<i_2<\ldots<i_m<i_{m+1}\}$ be any subset of cardinality
	$(m-1)$. Then, only on the lines of the form $\us{i\in A}{\bigcap}H_i$, there occurs a swap of points  \equ{\us{i\in A}{\bigcap}H_i\cap H_{j_m},\us{i\in A}{\bigcap}H_i\cap H_{j_{m+1}}} giving the required order of points on the corresponding lines of the arrangement $\ti{\mcl{H}}_n^m$. There is no change in the order of points on the remaining lines.  
\end{theorem}
\begin{proof}
	The proof of this theorem is immediate as the orientation of the simplex polyhedrality $\Gd^m H_{i_1}H_{i_2}\ldots H_{i_m}H_{i_{m+1}}$ corresponding to $1\leq i_1<i_2<\ldots<i_m<i_{m+1}\leq n$ changes.
\end{proof}

Now we prove a theorem which implies that the only possible distinct cones which give isomorphic generic hyperplane arrangements under isomorphisms which preserve subscripts are the antipodal 
pairs of cones.
\begin{theorem}
	\label{theorem:AntipodeCones}
	~\\
	Let $n>m>1$ be two positive integers. Let $\mcl{N}=\{L_1,L_2,\ldots,L_n\}$ be a normal system in $\mbb{R}^m$. Let $\mcl{C}^n_{\binom{n}{m+1}}$ be its associated discriminantal arrangement. 
	Let $\mcl{C}one^n_{\binom{n}{m+1}}$ be the set of conical convex regions of the discriminantal arrangement $\mcl{C}^n_{\binom{n}{m+1}}$. The interiors of 
	$C,D\in \mcl{C}one^n_{\binom{n}{m+1}}$ give rise to isomorphic generic hyperplane arrangements under an isomorphism which preserves subscripts if and only if $C=\pm D$. 
\end{theorem}
\begin{proof}
	One implication is clear.
	For every cone in $\mcl{C}one^n_{\binom{n}{m+1}}$ we associate a unique sign vector corresponding to each hyperplane of the discriminantal arrangement $\mcl{C}^n_{\binom{n}{m+1}}$. While moving from a cone to the adjacent cone through a hyperplane of $\mcl{C}^n_{\binom{n}{m+1}}$ there will be a swap of points as given in Theorem~\ref{theorem:CA} and sign of this hyperplane changes. If $C$ and $D$ are isomorphic cones, then while moving from $C$ to $D$ via hyperplanes there will be swap of points on lines and correspondingly there will be  sign changes of hyperplanes in $\mcl{C}^n_{\binom{n}{m+1}}$. Since $C$ and $D$ are isomorphic and the order of points on each line of corresponding hyperplane arrangements agree, we have that effectively the points on all lines undergo swapping or effectively no swapping occurs. If no swapping occurs the sign vector of cone $C$ matches with the sign vector of cone $D$. Hence $C=D$. In the other scenario, the sign vector of cone $C$ is the negative of the sign vector of cone $D$ and hence $C=\pm D$. This proves the theorem.
\end{proof}
Now we prove main Theorem~\ref{theorem:CountingInvariance} of the article.
\begin{proof}
	Using previous Theorem~\ref{theorem:AntipodeCones} there are exactly half of the number of convex cones of the discriminantal arrangement $\mcl{C}^n_{\binom{n}{m+1}}$ for any normal 
	system $\mcl{N}$ which give rise to distinct isomorphism classes. This proves main Theorem~\ref{theorem:CountingInvariance}.
\end{proof}
Here we prove a theorem about the number of isomorphism classes under isomorphisms which preserve subscripts for a normal system $\mcl{N}$ that is concurrency free.
\begin{theorem}
	\label{theorem:CountIsoClasses}
	Let $n>m>1$ be two positive integers. 
	Let \equ{\mcl{C}^n_{\binom{n}{m+1}}=\{M_{\{i_1<i_2<\ldots<i_m<i_{m+1}\}}\mid 1\leq i_1<i_2<\ldots<i_m<i_{m+1}\leq n\}} 
	be the discriminantal arrangement given by a normal system $\mcl{N}=\{L_1,L_2,\ldots,L_n\}$ which is concurrency free (refer to Definition~\ref{defn:ConcurrencyFree}) defined over $\mbb{R}$.
	Let $\mcl{E}=\{\{i_1,i_2,\ldots,i_m,$ $i_{m+1}\}\mid 1\leq i_1<i_2<\ldots<i_m<i_{m+1}\leq n\}$. Then the number of isomorphism classes of hyperplane arrangements
	under isomorphisms which preserve subscripts is given combinatorially by 
	\equ{\frac12 r(\mcl{C}^n_{\binom{n}{m+1}})=\frac{(-1)^n}{2}\gch(\mcl{C}^n_{\binom{n}{m+1}})(-1)=\frac 12\us{\mcl{D}\subs \mcl{E}}{\sum}(-1)^{n+\#(\mcl{D})+d_{\mcl{D}}}=\frac 12\us{\mcl{D}\subs \mcl{E}}{\sum}(-1)^{\#(\mcl{D})+rank(\overline{\mcl{D}})}}
	where 
	\begin{itemize}
		\item $\gch(\mcl{C}^n_{\binom{n}{m+1}})$ is the characteristic polynomial of the discriminantal arrangement $\mcl{C}^n_{\binom{n}{m+1}}$
		\item $n-d_\mcl{D}=rank(\overline{\mcl{D}})$ is the cardinality of a base collection $\ti{\mcl{D}}$ for $\overline{\mcl{D}}$.
	\end{itemize}
\end{theorem}
\begin{proof}
	This theorem follows from main Theorem~\ref{theorem:CountingInvariance} and Theorem~\ref{theorem:CharpolyCA}.
\end{proof}
\section{\bf{On the generic nature of the concurrency free condition of a normal system}}
\label{sec:gennature}
Here in this section we show that a generic normal system $\mcl{N}$ is concurrency free.
First we begin with an observation. 
\begin{note}
	\label{note:dimensioncuttingbyone}
	Let $V$ be a finite dimensional vector space of dimension $n$ over the field $\mbb{R}$. Let $0\neq W \subs V$ be a subspace and let $K\subs V$ be a co-dimension one subspace of $V$.
	If $W\nsubseteq K$ then we have $\dim(W\cap K)=\dim(W) +\dim(K)-\dim(W+K)=\dim(W)+n-1-n=\dim(W)-1$, that is, the dimension goes down by one when the subspace $W$ is 
	intersected with the co-dimension one space $K$ if $W\nsubseteq K$.
\end{note}
\subsection{Generic concurrency freeness of a normal system of six lines in a plane}
We begin with an example.
\begin{example}
	\label{example:degeneracyloci}
	Consider the example in Note~\ref{note:Nonconcurrencyfree} where we begin with six lines $L_i,1\leq i\leq 6$ in the plane $\mbb{R}^2$, no two are parallel, with the equation of the line $L_i$ given by $a_ix+b_iy=c_i, 1\leq i\leq 6$ ignoring
	the perpendicularity conditions. We derive conditions as to when the normal system is not concurrency free.
	Consider the concurrency closed sub-collections \equ{\mcl{D}_1=\{\{1,2,6\},\{1,3,5\},\{2,3,4\}\},\mcl{D}_2=\mcl{D}_1\cup\{\{4,5,6\}\}.} Using these concurrencies 
	we can obtain linear expressions for $c_4,c_5,c_6$ in terms of $c_1,c_2,c_3$ as follows.
	
	\equa{c_6&=\frac{a_6\vmattwo{c_1}{b_1}{c_2}{b_2}+b_6\vmattwo{a_1}{c_1}{a_2}{c_2}}{\vmattwo{a_1}{b_1}{a_2}{b_2}},\\
		c_5&=\frac{a_5\vmattwo{c_1}{b_1}{c_3}{b_3}+b_5\vmattwo{a_1}{c_1}{a_3}{c_3}}{\vmattwo{a_1}{b_1}{a_3}{b_3}},\\
	c_4&=\frac{a_4\vmattwo{c_2}{b_2}{c_3}{b_3}+b_4\vmattwo{a_2}{c_2}{a_3}{c_3}}{\vmattwo{a_2}{b_2}{a_3}{b_3}}.}
	
	If $L_4,L_5,L_6$ are concurrent we must have 
	\equan{det}{\vmatthree{a_4}{b_4}{c_4}{a_5}{b_5}{c_5}{a_6}{b_6}{c_6}=0 \Ra 
		\vmatthree{a_4}{b_4}{\frac{a_4\vmattwo{c_2}{b_2}{c_3}{b_3}+b_4\vmattwo{a_2}{c_2}{a_3}{c_3}}{a_2b_3-a_3b_2}}{a_5}{b_5}{\frac{a_5\vmattwo{c_1}{b_1}{c_3}{b_3}+b_5\vmattwo{a_1}{c_1}{a_3}{c_3}}{a_1b_3-a_3b_1}}{a_6}{b_6}{\frac{a_6\vmattwo{c_1}{b_1}{c_2}{b_2}+b_6\vmattwo{a_1}{c_1}{a_2}{c_2}}{a_1b_2-a_2b_1}}=0 } 
	The LHS expression of the last equation~\ref{Eq:det} after multiplying by \equ{(a_2b_3-a_3b_2)(a_1b_3-a_3b_1)(a_1b_2-a_2b_1)} is a linear polynomial in the variables $c_1,c_2,c_3$
	with coefficients as polynomials in the variables $a_i,b_i,1\leq i\leq 6$. Now the concurrency of $L_1,L_2,L_3$ implies concurrency of $L_4,L_5,L_6$, because in this case, all the six lines are concurrent. 
	Hence we obtain in addition that 
	
	\equa{&\text{the polynomial }\vmatthree{a_1}{b_1}{c_1}{a_2}{b_2}{c_2}{a_3}{b_3}{c_3} \text{ divides the polynomial or is a factor of }\\
		&(a_2b_3-a_3b_2)(a_1b_3-a_3b_1)(a_1b_2-a_2b_1)
		\vmatthree{a_4}{b_4}{\frac{a_4\vmattwo{c_2}{b_2}{c_3}{b_3}+b_4\vmattwo{a_2}{c_2}{a_3}{c_3}}{a_2b_3-a_3b_2}}{a_5}{b_5}{\frac{a_5\vmattwo{c_1}{b_1}{c_3}{b_3}+b_5\vmattwo{a_1}{c_1}{a_3}{c_3}}{a_1b_3-a_3b_1}}{a_6}{b_6}{\frac{a_6\vmattwo{c_1}{b_1}{c_2}{b_2}+b_6\vmattwo{a_1}{c_1}{a_2}{c_2}}{a_1b_2-a_2b_1}} }
	as a linear polynomial in $c_1,c_2,c_3$ with another polynomial factor \equ{g(a_1,a_2,a_3,a_4,a_5,a_6,b_1,b_2,b_3,b_4,b_5,b_6)} which does not involve $c_1,c_2,c_3$, and has only the variables $a_i,b_i,1\leq i\leq 6$.
	This polynomial $g$ gives a degeneracy condition on the coefficients $a_i,b_i;1\leq i\leq 6$ so that the concurrency $\{4,5,6\}$ occurs provided the concurrencies $\{1,2,6\},\{1,3,5\}$,  $\{2,3,4\}$ are already present.
	If $a_i,b_i;1\leq i\leq 6$ are so chosen such that the $12$-tuple $(a_1,a_2,a_3,a_4,a_5,a_6,b_1,b_2,b_3,b_4,b_5,b_6)$ does not lie on any of the finitely many degeneracy loci obtained from finitely many such 
	configurations of six lines, then the normal system $\mcl{N}=\{l_1,l_2,l_3,l_4,l_5,l_6\}$ consisting of six lines 
	\equ{l_i=\{t(a_i,b_i)\mid t\in \mbb{R},1\leq i\leq 6\}}
	in the plane $\mbb{R}^2$ is concurrency free. So generically a normal system of six lines in the plane is concurrency free.
\end{example}
\subsection{On generic normal systems}
~\\
In this section we prove a theorem about generic normal systems.
\begin{theorem}
	\label{theorem:generic}
	Let $n>m>1$ be two positive integers. Let \equ{\mcl{N}=\{L_i=\{t(a_{i1},a_{i2},\ldots,a_{im})\mid t\in \mbb{R}\},1\leq i\leq n\}} be a normal system in $\mbb{R}^m$. There exists a finite set 
	\equ{\mcl{F}=\{g_{u}\mid g_{u} \in \mbb{R}[A_{ij}:1\leq i\leq n,1\leq j\leq m], 1\leq u\leq N\}} of polynomials (degenerate loci) of cardinality $N\in \mbb{N}$ 
	such that if \linebreak $g\big([a_{ij}]_{1\leq i\leq n,1\leq j\leq m}\big)\neq 0$ for all $g\in \mcl{F}$, then the normal system $\mcl{N}$ is concurrency free.
\end{theorem}
\begin{proof}
	Let $\mcl{E}=\{\{i_1,i_2,\ldots,i_m,i_{m+1}\}\mid 1\leq i_1<i_2<\ldots<i_m<i_{m+1}\leq n\}$. After fixing the matrix $[a_{ij}]_{1\leq i\leq n,1\leq j\leq m}$ let 
	$\mcl{C}^n_{\binom{n}{m+1}}=\{M_{\{i_1,i_2,\ldots,i_m,i_{m+1}\}}\mid 1\leq i_1<i_2<\ldots<i_m<i_{m+1}\leq n\}$ be the discriminantal arrangement associated to the normal system $\mcl{N}$.
	Let $\mcl{D}\subs \mcl{E}$ be a concurrency closed sub-collection. For a $k$-hyperplane concurrency
	with $k>m$, we can solve for the constants of the equations of the $(k-m)$ hyperplanes. This will be linear in terms of the remaining constants. Now we go through the configuration imposing concurrency conditions such as the vanishing of the determinant as in Example~\ref{example:degeneracyloci}. Here we factor (if possible) these into a factor which is an irreducible linear constraint in terms of the variables 
	corresponding to constants and irreducible polynomials which give degeneracy loci.  Because of generic assumption on $\big([a_{ij}]_{1\leq i\leq n,1\leq j\leq m}\big)$ we conclude that the
	first factor is zero. Now we can solve for one constant coefficient reducing the number of independent variables by one for each such concurrency constraint. 
	
	This proves that for any concurrency closed set $\mcl{D}\subs \mcl{E}$ with concurrency orders $k_1,k_2,\ldots,k_r$ we have 
	\equa{\big(\us{i=1}{\os{r}{\sum}}k_i\big)-rm&=rank(\mcl{D})=\#\big(\mcl{D}'=Base\ of\ (\mcl{D})\big)\\&=n-\dim\bigg(\us{\{i_1,i_2,\ldots,i_m,i_{m+1}\}\in \mcl{D}}{\bigcap}M_{\{i_1<i_2<\ldots<i_m<i_{m+1}\}}\bigg).}
	This proves the theorem that the generic normal system $\mcl{N}$ considered here is concurrency free.
\end{proof} 
The following note is about the condition that the normal system $\mcl{N}$ and its associated discriminantal arrangement $\mcl{C}^n_{\binom{n}{m+1}}$
and should be concurrency free to compute the characteristic polynomial $\gch(\mcl{C}^n_{\binom{n}{m+1}})$.
\begin{note}
	\label{note:mildrestriction}
	~\\
	Let $M_{n\times m}(\mbb{R}) \sups \mcl{G}=\{[a_{ij}]_{1\leq i\leq n,1\leq j\leq m}\mid \mcl{N}=\{L_i=\{t(a_{i1},a_{i2},\ldots,a_{im})\mid t\in \mbb{R}\},1\leq i\leq n$ is concurrency free $\}$.
	Then $\mcl{G}$ contains a zariski open set $O$ which is zariski dense in $M_{n\times m}(\mbb{R})$. Hence the restriction to compute the characteristic
	polynomial $\gch(\mcl{C}^n_{\binom{n}{m+1}})$ for the discriminantal arrangement $\mcl{C}^n_{\binom{n}{m+1}}$ 
	that the normal system $\mcl{N}$ is concurrency free is a mild one. Moreover we can choose generic concurrency free
	normal systems defined over rationals (resp. integers) so that its discriminantal arrangement is also defined over rationals (resp. integers). Here finite field methods will also be applicable by reducing modulo 
	certain primes at which we have good reduction.
\end{note}
\section{\bf{Enumeration of Isomorphism Classes when the Normal System is not Concurrency Free via Two Examples}}
\label{sec:NonConcurrencyFree}
In this section we enumerate the isomorphism classes when the normal system is not
concurrency free in two examples.
\begin{example}[Example of Six Lines in the Plane]
	\label{Example:SixLines}
Consider the following generic six-line arrangement $\mcl{L}_6^2=\{L_1,L_2,L_3,L_4,L_5,L_6\}$ given by 
\equa{&L_1: x_1=0, L_2: 2x_1+3x_2=-2,L_3:3x_1+2x_2=3,,\\
	&L_4:x_2=0, L_5: 3x_1-2x_2=5,L_6: 2x_1-3x_2=5.}
We have $L_1\perp L_4,L_2\perp L_5,L_3\perp L_6$. This give rise to a normal system which is not concurrency free as in Note~\ref{note:Nonconcurrencyfree}.
The discriminantal arrangement of $\binom{6}{3}=20$ hyperplanes in $\mbb{R}^6$ are given by 
\equa{M_{\{1<2<3\}}&=\{(y_1,y_2,y_3,y_4,y_5,y_6)\in \mbb{R}^6\mid -5y_1-2y_2+3y_3=0\},\\ 
	M_{\{1<2<4\}}&=\{(y_1,y_2,y_3,y_4,y_5,y_6)\in \mbb{R}^6\mid 2y_1-y_2+3y_4=0\},\\
	M_{\{1<2<5\}}&=\{(y_1,y_2,y_3,y_4,y_5,y_6)\in \mbb{R}^6\mid -13y_1+2y_2+3y_5=0\},\\
	M_{\{1<2<6\}}&=\{(y_1,y_2,y_3,y_4,y_5,y_6)\in \mbb{R}^6\mid -12y_1+3y_2+3y_6=0\},\\
	M_{\{1<3<4\}}&=\{(y_1,y_2,y_3,y_4,y_5,y_6)\in \mbb{R}^6\mid 3y_1-y_3+2y_4=0\},\\
M_{\{1<3<5\}}&=\{(y_1,y_2,y_3,y_4,y_5,y_6)\in \mbb{R}^6\mid -12y_1+2y_3+2y_5=0\},}
\equa{M_{\{1<3<6\}}&=\{(y_1,y_2,y_3,y_4,y_5,y_6)\in \mbb{R}^6\mid -13y_1+3y_3+2y_6=0\},\\
	M_{\{1<4<5\}}&=\{(y_1,y_2,y_3,y_4,y_5,y_6)\in \mbb{R}^6\mid -3y_1+2y_4+y_5=0\},\\
	M_{\{1<4<6\}}&=\{(y_1,y_2,y_3,y_4,y_5,y_6)\in \mbb{R}^6\mid -2y_1+3y_4+y_6=0\},\\
M_{\{1<5<6\}}&=\{(y_1,y_2,y_3,y_4,y_5,y_6)\in \mbb{R}^6\mid -5y_1+3y_5-2y_6=0\},\\
M_{\{2<3<4\}}&=\{(y_1,y_2,y_3,y_4,y_5,y_6)\in \mbb{R}^6\mid 3y_2-2y_3-5y_4=0\},\\
	M_{\{2<3<5\}}&=\{(y_1,y_2,y_3,y_4,y_5,y_6)\in \mbb{R}^6\mid -12y_2+13y_3-5y_5=0\},\\
	M_{\{2<3<6\}}&=\{(y_1,y_2,y_3,y_4,y_5,y_6)\in \mbb{R}^6\mid -13y_2+12y_3-5y_6=0\},\\
	M_{\{2<4<5\}}&=\{(y_1,y_2,y_3,y_4,y_5,y_6)\in \mbb{R}^6\mid -3y_2+13y_4+2y_5=0\},\\
	M_{\{2<4<6\}}&=\{(y_1,y_2,y_3,y_4,y_5,y_6)\in \mbb{R}^6\mid -2y_2+12y_4+2y_6=0\},\\
	M_{\{2<5<6\}}&=\{(y_1,y_2,y_3,y_4,y_5,y_6)\in \mbb{R}^6\mid -5y_2+12y_5-13y_6=0\},\\
	M_{\{3<4<5\}}&=\{(y_1,y_2,y_3,y_4,y_5,y_6)\in \mbb{R}^6\mid -3y_3+12y_4+3y_5=0\},\\
	M_{\{3<4<6\}}&=\{(y_1,y_2,y_3,y_4,y_5,y_6)\in \mbb{R}^6\mid -2y_3+13y_4+3y_6=0\},\\
	M_{\{3<5<6\}}&=\{(y_1,y_2,y_3,y_4,y_5,y_6)\in \mbb{R}^6\mid -5y_3+13y_5-12y_6=0\},\\
	M_{\{4<5<6\}}&=\{(y_1,y_2,y_3,y_4,y_5,y_6)\in \mbb{R}^6\mid -5y_4+2y_5-3y_6=0\}}
and 
$\mcl{C}^6_{\binom{6}{3}}=\{M_{\{i<j<k\}}\mid 1\leq i<j<k\leq 6\}$. 
Now we use the following formula for the number of convex cones of a central arrangement.
For a central arrangement $\mcl{A}$ in $n$-dimensional Euclidean space, the number of convex cones in the arrangement is given by 
\equan{Regions}{(-1)^n\gch_{\mcl{A}}(-1)=(-1)^n\us{\mcl{B}\subseteq \mcl{A}}{\sum}(-1)^{\#(\mcl{B})+n-rank(\mcl{B})}.}
This formula can be derived from Proposition $3.11.3$ on Page $283$ of R.~Stanley~\cite{MR2868112}. We take here $\mcl{A}=\mcl{C}^6_{\binom{6}{3}}=B$ to simplify notation. Here $B$ denotes the $(20\times 6)$ matrix of coefficients of the variables $y_1,\ldots,y_6$ in the hyperplanes $M_{\{i<j<k\}}$, $1\leq i<j<k\leq 6$. Then we can compute the ranks of all submatrices corresponding to any finite subset $\mcl{B}$ of the rows of $B=\mcl{A}$, and use the above formula to compute $r^{\text{Non-Concurrency Free}}(\mcl{C}^6_{\binom{6}{3}})=(-1)^6\gch^{\text{Non-Concurrency Free}}_{\mcl{C}^6_{\binom{6}{3}}}(-1)$. Upon computation we obtain that there are \equ{(-1)^6\gch^{\text{Non-Concurrency Free}}_{\mcl{C}^6_{\binom{6}{3}}}(-1)=884 \text{ convex cones}.}

In this example we have the slopes $m_i$ of the lines $L_i,1\leq i\leq 6$ are given by $m_1=0,m_2=-\frac 23,m_3=-\frac 32,m_4=\infty,m_5=\frac32,m_6=\frac23$ with $m_2=\frac 1{m_3}$. If we consider another example of six lines $L_i,1\leq i\leq 6$ with slopes $m_1=0,m_2=-1,m_3=-2,m_4=\infty,m_5=1,m_6=\frac 12$ then we have $m_2\neq \frac 1{m_3}$ and a similar computation yields 888 convex cones of its associated discriminantal arrangment. 

Now the characteristic polynomial of a discriminantal arrangement corresponding to a concurrency free normal system is given by 
\equa{\gch^{\text{Concurrency Free}}_{\mcl{C}^6_{\binom{6}{3}}}(x)&=x^6-20x^5+145x^4-426x^3+300x^2\\
	&=x^2(x-1)(x^3-19x^2+126x-300).}
This can also be directly computed as in Section~\ref{sec:Examples}. (Otherwise we may refer to Y.~Numata, A.~Takemura~\cite{MR2986882} and H.~Koizumi, Y.~Numata, A.~Takemura~\cite{MR3000446}.)
Hence we have \equ{r^{\text{Concurrency Free}}(\mcl{C}^6_{\binom{6}{3}})=(-1)^6\gch^{\text{Concurrency Free}}_{\mcl{C}^6_{\binom{6}{3}}}(-1)=892.}
For a concurrency free normal system, there are $\frac{892-884}{2}=4$ more isomorphism classes which are not represented by any isomorphism class given by this normal system which is not concurrency free.

Now we will exhibit one such isomorphism class and give a theoretical proof.
\begin{figure}[h]
	\centering
	\includegraphics[width = 0.6\textwidth]{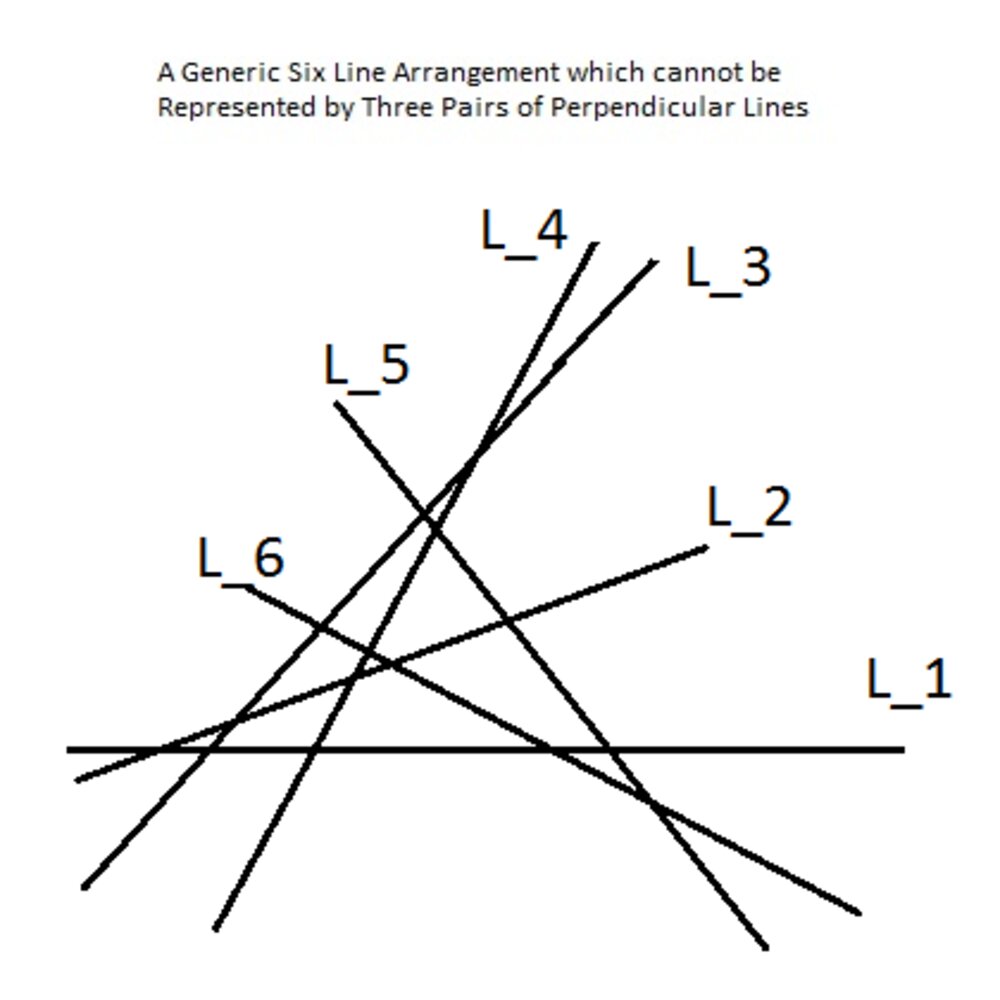}
	\caption{A Generic Six-Line Arrangements which cannot be represented Isomorphically by Three Pairs of Perpendicular lines.}
	\label{fig:Two}
\end{figure}

It is not a difficult exercise to show that the generic six-line arrangement given in Figure~\ref{fig:Two} cannot be represented isomorphically by three pairs of mutually perpendicular lines. (See Propositions~[3.3,3.4] in C.~P.~Anil Kumar~\cite{2011.05327}). We also know that any isomorphism class of a normal system which is not concurrency free can always be represented isomorphically by an isomorphism class arising from a concurrency free normal system by a density argument. Hence in this example we have proved theoretically that  
\equ{r^{\text{Concurrency Free}}(\mcl{C}^6_{\binom{6}{3}})>r^{\text{Non-Concurrency Free}}(\mcl{C}^6_{\binom{6}{3}}).}
\end{example}
\begin{example}[Example $3.2$ in M. Falk~\cite{MR1209098}]
\label{Example:SixPlanes}
Now we briefly mention another example arising from planes in $\mbb{R}^3$. 
Consider the following generic six plane arrangement  $\mcl{H}_6^2=\{H_1,H_2,\ldots,H_6\}$ in $\mbb{R}^3$ given by 
\equa{&H_1: x=0, H_2: y=0, H_3: z=0,\\ &H_4: 2x+2y+z=5, H_5: 2x+3y+2z=2, H_6: x+2y+2z=3.}
This gives rise to a normal system which is not concurrency free. Its associated discriminantal arrangement has 132 convex cones when computed in a similar manner as in Example~\ref{Example:SixLines}. 

Now the characteristic polynomial of a discriminantal arrangement corresponding to a concurrency free normal system is given by 
\equa{\gch^{\text{Concurrency Free}}_{\mcl{C}^6_{\binom{6}{4}}}(x)&=x^6-15x^5+69x^4-55x^3\\
	&=x^3(x-1)(x^2-14x+55).}
This can also be directly computed as in Section~\ref{sec:Examples} or refer to Y.~Numata, A.~Takemura~\cite{MR2986882} and H.~Koizumi, Y.~Numata, A.~Takemura~\cite{MR3000446}.
Hence we have \equ{r^{\text{Concurrency Free}}(\mcl{C}^6_{\binom{6}{4}})=(-1)^6\gch^{\text{Concurrency Free}}_{\mcl{C}^6_{\binom{6}{4}}}(-1)=140\neq 132.}
For a concurrency free normal system, there are $\frac{140-132}{2}=4$ more isomorphism classes which are not represented by any isomorphism class given by this normal system which is not concurrency free. 

\end{example}
We end the article with a final remark.
\begin{remark}
The initial few values of the number of convex cones of a discriminantal arrangement arising from a concurrency free normal system of n lines in $\mbb{R}^2$, for $3\leq n\leq 6$ are $2,8,62,892$ respectively. Hence the number of isomorphism classes of line arrangements for $3\leq n\leq 6$ are $1,4,31,446$ respectively.  We can extend the sequence of values for $n$ up to $10$ using the characteristic polynomial of discriminantal arrangements as given in Y.~Numata, A.~Takemura~\cite{MR2986882}.
\end{remark}
\section{\bf{Acknowledgement}}
I am grateful to the referee for making useful suggestions which enhanced the writing style of the article. 

\end{document}